\documentclass[12pt]{amsart}
\usepackage{amsfonts,amssymb}
\usepackage{amsmath,amscd}
\usepackage[all]{xy}
\newtheorem{thm}{Theorem}
 \newtheorem{cor}{Corollary}
 \newtheorem{lem}{Lemma}
 \newtheorem{prop}{Proposition}

 \theoremstyle{definition}
 \newtheorem{defn}{Definition}
 \newtheorem{ex}{Example}

\newcommand{\Hom}{\operatorname{Hom}}

\newcommand{\Ind}{\operatorname{Ind}}
\newcommand{\Res}{\operatorname{Res}}
\title[Two-parameter quantum algebras]{Two-parameter quantum algebras, canonical bases and categorifications}

\author{Zhaobing Fan and Yiqiang Li}
\address{Department of Mathematics\\ University at Buffalo, SUNY\\244 Mathematics Building
\\Buffalo, NY 14260}
\email{
zhaobing@buffalo.edu (Z.Fan),
yiqiang@buffalo.edu (Y.Li)
}
\date{\today}
\begin{document}
\begin{abstract}
A theory of canonical basis for a two-parameter quantum algebra is developed in parallel with the one in one-parameter case.
A geometric construction of the negative part of a two-parameter quantum algebra  is given by using mixed perverse sheaves and Deligne's weight theory based on Lusztig's work \cite{Lusztigbook}.
A categorification of the negative part of a two-parameter quantum algebra is provided.
A two-parameter quantum algebra is shown to be  a two-cocycle deformation, depending only on the second parameter,  of its one-parameter analogue.
%\\Keywords: Two-parameter quantum algebra, mixed perverse sheaf, weight, canonical basis, categorification.
\end{abstract}
\maketitle

\section{Introduction}

One of the landmarks in Lie theory is the theory of canonical basis for a one-parameter quantum algebra developed by Lusztig in the ADE case in ~\cite{Lusztigcanonical1},
and subsequently by Kashiwara \cite{Kashiwara91} and Lusztig ~\cite{Lusztigquiver, Lusztigbook} in general cases.
It serves many times as a source of inspirations for the creation of a new direction in Lie theory such as cluster algebras \cite{Fomin2002clus} and the categorification program \cite{CraneFrenkel}.

Among the various approaches to the  theory of one-parameter quantum algebras and  canonical bases, Lusztig's geometric construction of the negative part of a one-parameter quantum algebra by using perverse sheaves on representation varieties of a quiver plays a vital important role.
In his geometric setting, many algebraic objects have a very natural interpretation from which several hidden structures are revealed.
For example, the quantum parameter, the bar involution and  canonical basis elements are  incarnated as  the shift functor, the Verdier duality functor and simple perverse sheaves arising from the geometric setting, respectively.
The positivity of the structural constants of the canonical basis  follows naturally from this geometric setting.
 If one reads Lusztig's work carefully, one  notices that there is an ingredient, the Tate twist or the mixed structure, that Lusztig ignored in his geometric framework (see ~\cite[8.1.4]{Lusztigbook}).  It is desirable to see what Lusztig's geometric framework provides if the Tate twist is added.

In this paper, we construct an algebra from the mixed version of Lusztig's geometric framework by using mixed perverse sheaves on representation varieties of a quiver  and Deligne's theory of weight, and we show that this algebra is isomorphic to the negative part of a two-parameter quantum algebra,
in which the Tate twist corresponds to the second parameter.

From  this geometric construction,  we obtain several new features of a two-parameter quantum algebra.
We are able to get a new presentation of generators and relations for a two-parameter quantum algebra determined by a certain matrix. This matrix serves as the generalized Cartan matrix and its symmetrization  in the  definition of a one-parameter quantum algebra. It is determined by a chosen orientation of a graph in symmetric cases.  This presentation is new even in finite type. For example, the two parameters $v$ and $t$
we used in this paper are different from the one $(\alpha, \beta)$ used in literature in that they are related by $\alpha= vt$ and $\beta=vt^{-1}$.
Furthermore, this presentation covers all Kac-Moody cases, unlike the one in literature which mainly studies finite type and some affine types.
More importantly, it provides a new connection between  a one-parameter quantum algebra and  a two-parameter quantum algebra.
As is shown in this paper,  a two-parameter quantum algebra is a two-cocycle deformation, depending only on the second parameter, of its one-parameter analogue. As a consequence, if the underlying Cartan matrices of two two-parameter quantum algebras are the same, then they must be deformations of each other, and the deformation only depends on the second parameter. Last but not least, from the new presentation, we obtain a categorification of the negative part by utilizing Khovanov-Lauda's work ~\cite{Khovanov2009diag} and ~\cite{Khovanov2010diagrammatic}.

From  the geometric setting, we  also obtain a basis for the negative part of a two-parameter quantum algebra consisting of   simple perverse sheaves $of$ $weight$ $zero$. If one forgets the Tate twist, this basis is exactly the canonical basis in the one-parameter case.
Moreover, the basis is a deformation of the canonical basis in the one-parameter case.
In addition to its compatibility with the canonical basis in the one-parameter case, this basis admits many favorable properties such as integrality and positivity as does its one-parameter analogue. It also gives rise to a basis for each irreducible integrable highest weight module simultaneously.
Moreover, if the underlying Cartan matrices of two two-parameter quantum algebras are the same, the canonical bases coincide under the deformation from one algebra to another (see Corollary ~\ref{cor6}).
We follow Lusztig's approach in one-parameter case to give an algebraic characterization of this basis, and we call it the canonical basis of the negative half of a  two-parameter quantum algebra. The characterization is in complete analogy with the one in one-parameter case.
In particular, up to a sign, it is characterized by three properties:
it is in the integral form of the negative half, it is bar invariant and it is almost orthonormal with respect to a bilinear form.
This characterization is made possible by identifying the  negative part with an analogue of Lusztig's algebra $\mathbf f$, which again comes from the geometric construction.
In particular, both the bar involution and  the bilinear form  have natural interpretations in the geometric framework.
The process to get rid of the sign is completely algebraic and follows closely Lusztig's argument in the one-parameter case.

In short, the mixed version of  Lusztig's geometric framework  is a natural geometric setting to study two-parameter quantum algebras.
It provides  a geometric construction of the negative part of a two-parameter quantum algebra.
Based on the geometric work, we develop a canonical basis theory for the negative part  of a two-parameter quantum algebra in an approach parallel with Lusztig's approach in the one-parameter case. Moreover, we show that the two-parameter quantum algebra is a two-cocycle deformation, depending only on the second parameter, of its one-parameter analogue.
Finally, we give a categorification of the negative part  in the sense of  ~\cite{Khovanov2009diag}.

The intimate relationship revealed in this paper  between a two-parameter quantum algebra and its one-parameter analogue by the  specialization at $t=1$ and the deformation  should play an important role in a forthcoming paper,
 where we shall continue to develop the canonical basis theory for the tensor product of integrable representations of two-parameter quantum algebras and two-parameter analogues of Lusztig's modified quantum algebras $\dot{\mathbf U}$.

A similar relationship between a quantum super algebra and the related two-parameter  quantum algebra   by a specialization at $t=\pm \mathbf{i}$, the imaginary unit, and  a deformation with respect to the second parameter  will be  elaborated in ~\cite{Clark-Fan-Li-Wangtwo}.
Among others, we will provide a new  categorification of a  quantum super algebra different from the one in ~\cite{HillWang}.
This study, combining with the results in this paper and ~\cite{Clark-Fan-Li-Wangtwo}, should lead to interesting relations in the structural and representation theories of a one-parameter quantum algebra, its super  and  two-parameter analogues.

Meanwhile, the results obtained in this paper  strongly suggest that a theory of crystal basis for two-parameter quantum algebras can be developed
in parallel with the one in the one-parameter case by Kashiwara.
We also  hope that our work on two-parameter quantum algebras can  shed some light on the open problem
to develop a canonical basis theory  for multiparameter quantum algebras.

Finally we remark that two-parameter quantum algebras have been studied from the early 1990s   by various authors, see
\cite{D92, Takeuchi, benkart2001rep, Reineke2001,Christian2004, BergeronGaoHu2007, Humultipara, JingZhang2011,hu2012notes, KC09} and the references therein.
They also appear in I.B  Frenkel's philosophical observations on the interactions of affinizations and
quantizations of a Lie algebra (\cite{Igorwork}).
 A new exciting development  is Hill and Wang's categorification of the  covering quantum algebra $\mathbf f^{\pi}$  which has two parameters with the second parameter $\pi$ subject to  $\pi^2=1$ in ~\cite{HillWang}. These work also inspire us during the formation of this paper.  In Section ~\ref{Comparison}, we compare the two-parameter quantum Serre relations with those available in the literature.

This paper is organized as follows.
In Section 2, we review the geometric background, perverse sheaves and weight theory.
We  construct the algebra $\mathfrak{K}$ which is a geometric realization of $\mathfrak{f}$ for symmetric cases.
In Section 3, we algebraically construct the algebra $\mathfrak{f}$, a two-parameter analogue of Lusztig's algebra $\mathbf{f}$,
and compare it with various algebras in literatures. Those who are not interested in geometry can read this section directly.
Section 4 provides two relations between the algebra $\mathfrak{f}$ and Lusztig's algebra $\mathbf{f}$ by specialization and deformation.
These relations are generalized to the entire algebras.
In Section 5, we present the algebraic characterization of the canonical basis of $\mathfrak{f}$,
as well as that of the irreducible highest weight  $U_{v,t}$-module $L(\lambda, \epsilon)$.
Meanwhile, we show that the canonical basis of $\mathfrak{f}$ gets identified
with the set of simple perverse sheaves of weight zero appeared in the geometric construction.
In Section 6, we give an algebraic  categorification of $\mathfrak{f}$ which covers all symmetrizable cases.

\subsection*{Acknowledgements}
We learned from Weiqiang Wang that two-parameter quantum algebras should be able to be realized geometrically. We thank Weiqiang for sharing with us his great idea and numerous discussions and comments on this project.
We thank Zongzhu Lin for his valuable comments.
We thank Jonathan Kujawa for bringing  the paper \cite{Frenkel2011langlands} to our attention.
Y. Li is partially supported by the NSF grant:  DMS 1160351.

\setcounter{tocdepth}{1}
\tableofcontents

\section{The algebra $\mathfrak{K}$}\label{sec3}

\subsection{Review of mixed perverse sheaves}\label{sec3.1}

We review briefly the theory of mixed perverse sheaves. We
refer to Chapter 8 in \cite{Lusztigbook} and
\cite{BBD,FK} for more details.

Let $k$ be an  algebraic closure of  a finite field of $q$ elements.
All algebraic varieties considered in this paper are over $k$. Let $l$ be a fixed
prime number which is invertible in $k$, and
$\overline{\mathbb{Q}}_l$ be an algebraic closure of the field
$\mathbb{Q}_l$ of $l$-adic numbers. Denote by
$\mathcal{D}(X)$ %$=\mathcal{D}^b_c(X)$
the bounded derived category of
$\overline{\mathbb{Q}}_l$-constructible sheaves on the algebraic variety $X$. Let
$\mathcal{M}(X)$ be the full subcategory of $\mathcal{D}(X)$ consisting of perverse sheaves on $X$.
 We denote by ${\bf 1}_X$ the constant sheaf $\overline{\mathbb{Q}}_l$ on $X$, and simply by ${\bf 1}$ if $X$ is obvious from the context.

Given any integer $n$, let $[n]: \mathcal{D}(X) \rightarrow
\mathcal{D}(X)$ be the shift functor and $(n): \mathcal{D}(X) \rightarrow
\mathcal{D}(X)$ be the Tate twist functor.
Let ${}^p\!H^n: \mathcal{D}(X) \rightarrow \mathcal{M}(X)$ be the perverse cohomology functor,
and $\mathbb{D}: \mathcal{D}(X) \rightarrow \mathcal{D}(X)$ be the Verdier dual functor.
Let $f: X \rightarrow Y$ be a morphism of algebraic varieties. There are functors
 $f^*, f^!: \mathcal{D}(Y) \rightarrow \mathcal{D}(X)$ and $f_*, f_!:
\mathcal{D}(X) \rightarrow \mathcal{D}(Y)$.
Moreover, if $f: X \rightarrow Y$ is a locally trivial
principal $G$-bundle, then there is a well-defined functor $f_{\flat}: \mathcal{M}_G(X)[n] \rightarrow \mathcal{M}(Y)[n+d]$ defined by $f_{\flat}(K)={}^p\!H^{-n-d}(f_*K)[n+d]$, where $\mathcal{M}_G(X)$ is the full subcategory of $\mathcal{M}(X)$ consisting of all $G$-equivariant perverse sheaves and $d=\dim G$.

 Let $\mathcal{D}_m(X)$ be the full subcategory of $\mathcal{D}(X)$ consisting of all mixed complexes, and $\mathcal{D}_{\leq w}(X)$ (resp. $\mathcal{D}_{\geq w}(X)$)  be the full subcategory of $\mathcal{D}_m(X)$ consisting of all complexes whose $i$-th cohomology has weight $\leq w+i$ (resp. $\geq w+i$). We simply denote by $\mathcal{D}_{\leq w}$ (resp. $\mathcal{D}_{\geq w}$ ) instead of $\mathcal{D}_{\leq w}(X)$ (resp. $\mathcal{D}_{\geq w}(X)$) if $X$ is obvious from the context.

A complex $K \in \mathcal{D}_m(X)$ is called pure of weight $w$ if $K \in \mathcal{D}_{\leq w}(X)\bigcap \mathcal{D}_{\geq w}(X)$. Denote by ${\rm wt}(K)$ the weight of a pure complex $K$.

The functors $f^*, f_*, f^!, f_!, [j], \otimes$ and Tate twist $(n)$ send mixed complexes to  mixed complexes.
We list some more properties as follows.

\begin{itemize}
\item[(a)] Simple perverse sheaves are pure.

\item[(b)] If $X$ is smooth, then $\mathbf 1_X$ is pure of weight 0.

\item[(c)] If $K$ is a pure complex, then ${\rm wt}(K[1])={\rm wt}(K)+1, {\rm wt}(K(1))={\rm wt}(K)-2$.

\item[(d)] $\mathbb{D}(K[j])=\mathbb{D}(K)[-j]$, \ $\mathbb{D}(K(n))=\mathbb{D}(K)(-n)$, \  ${}^p\!H^n(K)={}^p\!H^0(K[n])$.

  \item[(e)] $\mathbb{D}(\mathcal{D}_{\leq w}) \subset \mathcal{D}_{\geq -w}$ and $\mathbb{D}(\mathcal{D}_{\geq w}) \subset \mathcal{D}_{\leq -w}$. In particular, the Verdier dual sends pure complexes of weight $w$ to pure complexes of weight $-w$.

  %A subquotient of any object in $\mathcal{D}_{\leq w}$ (resp. $\mathcal{D}_{\geq w}$) is still in  $\mathcal{D}_{\leq w}$ (resp. $\mathcal{D}_{\geq w}$).

  \item[(f)] The external tensor product functor $\boxtimes$ sends $\mathcal{D}_{\leq w_1} \times \mathcal{D}_{\leq w_2}$ (resp. $\mathcal{D}_{\geq w_1} \times \mathcal{D}_{\geq w_2}$) to $\mathcal{D}_{\leq w_1+w_2}$ (resp. $\mathcal{D}_{\geq w_1+w_2}$). In particular, if $K,L$ are pure complexes, then ${\rm wt}(K\boxtimes L)={\rm wt}(K)+{\rm wt}(L).$

  \item[(g)] If $f: X\rightarrow Y$ is a morphism of varieties, then $f^*$ and $f_!$ preserve $\mathcal{D}_{\leq w}$  and $f_*$ and  $f^!$ preserve $\mathcal{D}_{\geq w}$. In particular, if $f$ is a proper map, then
      $f_!$ sends pure complexes of weight $w$ to pure complexes of weight $w$.

  \item[(h)] If $f: X \rightarrow Y$ is smooth with connected fibers of fiber dimension
$d$, then
$f^*[d]=f^![-d](-d)$
and $\mathbb{D}f^*(L)=f^!(\mathbb{D}L)$. Moreover, ${\rm wt}(f^*K)={\rm wt}(K)$ for any pure complex $K$.
\end{itemize}

\subsection{The matrix $\Omega$}\label{sec2.1}

Let $I$ be a finite set. Throughout this paper, we fix a matrix $\Omega=(\Omega_{ij})_{i,j\in I}$ satisfying that
\begin{itemize}
  \item[(a)] $\Omega_{ii} \in \mathbb{Z}_{>0}$, $\Omega_{ij}\in \mathbb{Z}_{\leq 0}$ for all $i\neq j \in I$;
  \item[(b)] $\frac{\Omega_{ij}+\Omega_{ji}}{\Omega_{ii}}\in \mathbb{Z}_{\leq 0}$ for all $i\neq j \in I$;
  \item[(c)] the greatest common divisor of all $\Omega_{ii}$ is equal to 1.
\end{itemize}
To $\Omega$, we associate the following three bilinear forms on $\mathbb{Z}^I$.
\begin{eqnarray}
\langle i, j\rangle &=&
\Omega_{ij}, \quad \hspace{45pt}\forall i, j\in I. \label{eq47}\\
\begin{bmatrix} i,j \end{bmatrix}&=& 2\delta_{ij} \Omega_{ii} -\Omega_{ij}, \quad \forall i, j\in I. \label{eq48}\\
i\cdot j&=&\langle i, j\rangle +\langle j,i\rangle,  \quad \forall i, j\in I. \label{eq49}
\end{eqnarray}
Note that the form $``\cdot"$ satisfies the following properties:
$$i \cdot i \in 2\mathbb{Z}_{>0}\ {\rm for\ any}\ i \in I\ {\rm and}\ 2\frac{i \cdot j}{i\cdot i} \in \mathbb{Z}_{\leq 0}\ {\rm for\ any}\ i\neq j\ {\rm in}\ I.$$
It is a Cartan datum in Section 1.1.1 in \cite{Lusztigbook}.

The matrix $\Omega$ is called of  symmetric type if $\Omega_{ii}=1,\ \forall i\in I$.
In this case,  the associated  Cartan datum is of symmetric type.

{\it For simplicity, we assume that $\Omega$ is of symmetric type in the rest of this section.}
To such a matrix, we associate a quiver whose vertex set is $I$, and whose arrow set consisting of $-\Omega_{ij}$ many arrows from vertex $i$ to vertex $j$ if $i\not =j$.  By an abuse of notation, we  denote by $\Omega$ the associated quiver.  Since the matrix $\Omega$ is fixed, the quiver is thus fixed.

Note that the assignment of sending a matrix  to its associated quiver defines a bijection between the set of such matrices and the set of quivers, up to isomorphisms.
% In this section, we give a geometric construction of ${}_{\mathfrak{A}}\mathfrak{f}$ in a modification of Lusztig's framework \cite{Lusztigbook}.

\subsection{The category $\mathfrak{Q}^m_V$}\label{sec3.3}

Let $V=\bigoplus_{i \in I}V_i$ be  an $I$-graded $k$-vector space and $\underline{\dim}V=(\dim V_i)_{i \in I}\in \mathbb{N}^I$.
We define
\begin{equation}\label{eq1}
  E_V=\bigoplus_{h \in \Omega} \Hom(V_{h'}, V_{h''}), \quad  G_V=\bigoplus_{i\in I} GL(V_i),
\end{equation}
where $h'$ and $h''$ are  the source and target of the  arrow $h$ in $\Omega$, respectively.
$G_V$ acts on $E_V$ by
conjugation, i.e., $gx=x'$ and $x'_h=g_{h''}x_hg^{-1}_{h'}$ for all
$h \in \Omega$.

A subset $I'$ in $I$ is said to be {\it discrete} if there is no arrow $h \in \Omega$ such that $\{h', h''\} \subset I'$. We set ${\rm supp}(\nu)=\{i \in I\ |\ \nu_i \neq 0\}$, for any $\nu \in \mathbb{N}^I$. We call $\nu \in \mathbb{N}^I$ discrete if ${\rm supp}(\nu)$ is discrete.

Let $\underline{\nu}=(\nu^1, \nu^2,\cdots, \nu^m)$ be a sequence in $\mathbb{N}^I$ such that $\sum_{1\leq l\leq m}\nu^l_i=\dim V_i$ and  $\nu^l$ is discrete for all $l=1,\cdots,m$.
A {\it flag of type}
$\underline{\nu}$ in  $V$ is a
sequence
$$f=(V=V^0\supset V^1 \supset \cdots \supset V^m=0)$$
of $I$-graded vector spaces such that $\underline{\dim}V^{l-1}/V^l=\nu^l,\ \forall 1\leq l\leq m$.
Let $\mathcal{F}_{\underline{\nu}}$ be the
$k$-variety of all flags of type $\underline{\nu}$ in $V$.
Let $\widetilde{\mathcal{F}}_{\underline{\nu}}=\{(x,f) \in E_V \times
\mathcal{F}_{\underline{\nu}} \mid$ $f$ is
$x$-stable$\}$, where $f$ is $x$-stable if $x_h(V^l_{h'}) \subset
V^l_{h''}$, for all $h \in \Omega$ and $1\leq l\leq m$.

Let $G_V$ act on $\mathcal{F}_{\underline{\nu}}$ by $g\cdot f \mapsto gf$, where
$gf=(gV^0\supset gV^1 \supset \cdots \supset gV^m=0),$ if $f=(V=V^0\supset V^1 \supset \cdots \supset V^m=0)$.
Let $G_V$ act diagonally on $\widetilde{\mathcal{F}}_{\underline{\nu}}$, i.e., $g\cdot (x,f) \mapsto
(gx,gf)$.
By Proposition 9.1.3 in \cite{Lusztigbook}, %~\cite[Proposition 9.1.3]{Lusztigbook},
 we have that  $\widetilde{\mathcal{F}}_{\underline{\nu}}$ is a
smooth irreducible variety of dimension
\begin{equation}\label{eq102}
d(\underline{\nu})=\sum_{i,l<l'}\nu^{l'}_i \nu^l_i +\sum_{h, l'<l} \nu^{l'}_{h'}\nu^l_{h''}.
\end{equation}
Moreover, the first projection map
$\pi_{\underline{\nu}}: \widetilde{\mathcal{F}}_{\underline{\nu}} \rightarrow E_V$ is a proper $G_V$-equivariant morphism.
As a consequence, the complex $(\pi_{\underline{\nu}})_!{\bf 1}_{\widetilde{\mathcal{F}}_{\underline{\nu}}}$ is a
semisimple complex.

\begin{lem}\label{lem1}
The complex $\widetilde{L}_{\underline{\nu}}=
(\pi_{\underline{\nu}})_!{\bf 1}_{\widetilde{\mathcal{F}}_{\underline{\nu}}}$ is pure of weight 0.
\end{lem}
\begin{proof}
  The lemma follows from (b) and Section \ref{sec3.1}(g).
\end{proof}
We set
\begin{equation}\label{eq3}
  \mathfrak{L}_{\underline{\nu}}=\widetilde{L}_{\underline{\nu}}[d(\underline{\nu})](d(\underline{\nu})).
\end{equation}

Let $\mathfrak{Q}^m_V$ be the full subcategory of $\mathcal{D}_m(E_V)$
whose objects are isomorphic to finite direct sums of $L[d](n)$ for various $d \in
\mathbb{Z}$, $n\in \frac{1}{2}\mathbb{Z}$ and
various simple perverse sheaves $L $ satisfying the following property:
$L$ is a direct summand of $\widetilde{L}_{\underline{\nu}}$ up to a shift and a Tate twist for
 some $\underline{\nu} \in \mathbb{N}^I$ such that $\sum_{1\leq l\leq m}\nu^l_i=\dim V_i$.

 We set $\mathfrak{Q}^{\leq w}_V=\mathfrak{Q}^m_V \bigcap \mathcal{D}_{\leq w}(E_V)$. This is the full subcategory of $\mathfrak{Q}^m_V$ consisting of mixed complexes whose $i$-th cohomology sheaf has weight $\leq w+i$ for all $i\in \mathbb{Z}$. Similarly, let $\mathfrak{Q}^{\geq w}_V=\mathfrak{Q}^m_V \bigcap \mathcal{D}_{\geq w}(E_V)$. We notice that $\mathfrak{Q}^{\leq 0}_V\bigcap \mathfrak{Q}^{\geq 0}_V$ is the same as $\mathcal{Q}_V$ defined in \cite{Lusztigbook}.

 \subsection{Additive generators}

Let $\mathfrak{K}_V$ be the split
Grothendieck group of the category $\mathfrak{Q}^m_V$. More precisely,
 $\mathfrak{K}_V$ is the abelian group generated by the isomorphism classes of objects in $\mathfrak{Q}^m_V$ which subjects to the following relation:
\begin{equation}\label{eq90}
  [L\oplus L']=[L]+[L'], \quad \forall L,L'\in \mathfrak{Q}^m_V.
\end{equation}
Let $\mathfrak{M}_V$ be the split Grothendieck group of the full subcategory of $\mathfrak{Q}^m_V$ which
consists of all direct sums of $\mathfrak{L}_{\underline{\nu}}$
for various $\underline{\nu}$ up to shifts and Tate twists.
 Similarly, let $\mathfrak{K}^{\leq w}_V$ (resp. $\mathfrak{K}^{\geq w}_V$) be the split
Grothendieck group of the category $\mathfrak{Q}^{\leq w}_V$ (resp. $\mathfrak{Q}^{\geq w}_V$).

By an abuse of notation, we write $L$ instead of $[L]$ for elements in the Grothendieck group. Let $v$ and $t$ be two independent indeterminates and $\mathfrak{A}=\mathbb{Z}[v^{\pm 1}, t^{\pm 1}]$ be the subring of Laurent polynomials in  $\mathbb{Q}(v,t)$. We define an $\mathfrak{A}$-action on $\mathfrak{K}_V$ by
\begin{eqnarray}\label{eq10}
  v \cdot  L=L[1](\frac{1}{2}),\quad
  t \cdot L =L(\frac{1}{2}).
\end{eqnarray}
 Then $\mathfrak{K}_V$ is
an $\mathfrak{A}$-module generated by the simple perverse sheaves of weight 0 in $\mathfrak{Q}^m_V$. Moreover, $\mathfrak{M}_V$ is
an  $\mathfrak{A}$-submodule of $\mathfrak{K}_V$ generated by
$\mathfrak{L}_{\underline{\nu}}$.

\begin{thm}\label{thm1}
$\mathfrak{M}_V = \mathfrak{K}_V$ as $\mathfrak{A}$-modules,
 i.e., the set of $\mathfrak{L}_{\underline{\nu}}$ for various $\underline{\nu}$
  contains an $\mathfrak{A}$-basis of $\mathfrak{K}_V$.
\end{thm}

\begin{proof} Recall that $\mathfrak{Q}^{\leq 0}_V\bigcap \mathfrak{Q}^{\geq 0}_V$ is the full subcategory of $\mathfrak{Q}^m_V$ consisting of pure complexes of weight 0 and ${\rm wt}(\widetilde{L}_{\underline{\nu}})=0$ for any $\underline{\nu}$.
By Proposition 12.6.2 in \cite{Lusztigbook},  $\widetilde{L}_{\underline{\nu}}$
are additive generators of $\mathfrak{K}^{\leq 0}_V \bigcap \mathfrak{K}^{\geq 0}_V$.
Furthermore, by the definition of $\mathfrak{Q}^m_V$, for any element $K \in \mathfrak{K}_V$, there exist $B_s \in \mathfrak{K}^{\leq 0}_V \bigcap \mathfrak{K}^{\geq 0}_V$ such that $K=\sum_s v^{n_s}t^{m_s}B_s$ for some $n_s, m_s$.
This implies that $\widetilde{L}_{\underline{\nu}}$
 are additive generators of $\mathfrak{K}_V$. Moreover, $\mathfrak{L}_{\underline{\nu}}$
 are also additive generators of $\mathfrak{K}_V$. This proves the theorem.
\end{proof}

\subsection{Induction functor}\label{ind}
Let $W$ be an $I$-graded subspace of $V$ and $T=V/W$. Let $F=\{x\in E_V\ |\ x(W)\subset W\}$. For any $x\in F$, let $x_W$ be the restriction of $x$ to $W$ and $x_T: V/W \rightarrow V/W$ be the induced map of $x$ by passage to the quotient. Let $P$
be the stabilizer of $W$ in $G_V$ and $U$ be its unipotent radical.
 Consider Lusztig's diagram
 \begin{equation}\label{eq4}
\xymatrix{ E_T\times E_W & G_V \times^U F \ar[l]_-{p_1}\ar[r]^-{p_2}&
G_V \times^P F \ar[r]^-{p_3} &E_V,}
 \end{equation}
where $p_1(g,x)=(x_T, x_W), %\kappa(x),\
p_2(g,x)=(g,x),  p_3(g,x)=g(\iota(x))$ and $\iota: F \rightarrow E_V$ is the embedding map.
Let $G_T \times G_W$ act on $E_T \times E_W$ component-wise.
We define
$$\widetilde{\Ind}^V_{T,W}K=p_{3!}p_{2\flat}p^*_1K, \quad \forall K \in
\mathcal{D}_{G_T \times G_W}(E_T \times E_W).$$

\begin{prop}\label{prop2} If $K\in \mathfrak{Q}_T$ and $L \in \mathfrak{Q}_W$, then $\widetilde{\Ind}^V_{T,W}(K\boxtimes L) \in \mathfrak{Q}_V$. Moreover, if both $K$ and $L$ are pure, so is $\widetilde{\Ind}^V_{T,W}(K\boxtimes L)$, and its weight is equal to the sum of the weights of $K$ and $L$.
%$${\rm wt}(\widetilde{\Ind}^V_{T,W}(K\boxtimes L))={\rm wt}(K)+{\rm wt}(L).$$
\end{prop}
\begin{proof} The proof of the first part of the proposition is similar to that of Lemma 9.2.3 in \cite{Lusztigbook}. We only need to show the second part of the proposition.
By Section \ref{sec3.1}(f) and (h), we have
$${\rm wt}(p_{2\flat}p_1^*(K \boxtimes L))={\rm wt}(K)+{\rm wt}(L).$$
The proposition follows from the fact that $p_3$ is a proper map and Section \ref{sec3.1}(g).
\end{proof}
We set
\begin{equation}\label{eq5}
\mathfrak{Ind}^V_{T,W}(K\boxtimes L)=\widetilde{\Ind}^V_{T,W}(K\boxtimes L)[d_1-d_2](d_1-d_2),
\end{equation}
where $d_1$ (resp. $d_2$) is the fiber dimension of $p_1$ (resp. $p_2$) in Diagram (\ref{eq4}).
%By Proposition \ref{prop2} and the definition of $\mathfrak{Ind}^V_{T,W}$, we have the following corollary.
%\begin{cor}\label{cor1}
%If both $K$ and $L$ are pure, then
%$${\rm wt}(\mathfrak{Ind}^V_{T,W}(K\boxtimes L))={\rm wt}(K)+{\rm wt}(L)-(d_1-d_2).$$
%\end{cor}
\begin{prop}\label{prop3}
\begin{itemize}
\item[(a)] If both $K$ and $L$ are pure, then
$${\rm wt}(\mathfrak{Ind}^V_{T,W}(K\boxtimes L))={\rm wt}(K)+{\rm wt}(L)-(d_1-d_2).$$
\item[(b)] $\mathfrak{Ind}^V_{T,W}(\mathfrak{L}_{\underline{\nu}'} \boxtimes \mathfrak{L}_{\underline{\nu}''})=\mathfrak{L}_{\underline{\nu}'\underline{\nu}''},$ where $\mathfrak{L}_{\underline{\nu}}$ is defined in (\ref{eq3}).
\end{itemize}
\end{prop}
\begin{proof} Part (a) follows from Proposition \ref{prop2} and (\ref{eq5}). We now prove part (b).
\begin{equation*}
\begin{split}
   \mathfrak{Ind}&^V_{T,W}(\mathfrak{L}_{\underline{\nu}'} \boxtimes \mathfrak{L}_{\underline{\nu}''})
   = \widetilde{\Ind}^V_{T,W}(\widetilde{L}_{\underline{\nu}'} \boxtimes \widetilde{L}_{\underline{\nu}''})[d(\underline{\nu}')+d(\underline{\nu}'')+d_1-d_2]
   (d(\underline{\nu}')+d(\underline{\nu}'')+d_1-d_2)\\
   &= \widetilde{L}_{\underline{\nu}'\underline{\nu}''}[d(\underline{\nu}')+d(\underline{\nu}'')
   +d_1-d_2](d(\underline{\nu}')+d(\underline{\nu}'')+d_1-d_2)\\
   &= \mathfrak{L}_{\underline{\nu}'\underline{\nu}''}[d(\underline{\nu}')+d(\underline{\nu}'')+d_1-d_2-
   d(\underline{\nu}'\underline{\nu}'')]
   (d(\underline{\nu}')+d(\underline{\nu}'')+d_1-d_2-d(\underline{\nu}'\underline{\nu}'')).
\end{split}
\end{equation*}
Part (b) follows from the fact that $d(\underline{\nu}')+d(\underline{\nu}'')+d_1-d_2=d(\underline{\nu}'\underline{\nu}'').$
\end{proof}

\subsection{The algebra $(\mathfrak{K}, \mathfrak{Ind})$}\label{sec3.8}
 We notice that if $\underline{\dim}V_1=\underline{\dim}V_2$, then $E_{V_1}$ and $E_{V_2}$ are isomorphic.
Moreover, the categories $\mathfrak{Q}^m_{V_1}$ and $\mathfrak{Q}^m_{V_2}$ are isomorphic.
So we write $\mathfrak{Q}^m_{\nu}$ (resp. $\mathfrak{K}_{\nu}$) instead of $\mathfrak{Q}^m_{V}$ (resp. $\mathfrak{K}_{V}$) if $\underline{\dim}V=\nu$.
Now let $$\mathfrak{K}=\oplus_{\nu \in \mathbb{N}^I} \mathfrak{K}_{\nu}.$$
Define a multiplication on $\mathfrak{K}$ as follows.
 \begin{align*}
\mathfrak{Ind}: \mathfrak{K}
\times \mathfrak{K} &\rightarrow \mathfrak{K},\quad
 (K,L) \mapsto
\mathfrak{Ind}^{\nu}_{\tau, \omega}(K\otimes L),
 \end{align*}
for any homogenous elements $K,L$ with $K \in \mathfrak{K}_{\tau}$ and
$L \in \mathfrak{K}_{\omega}$.
\begin{thm}\label{thm2}
{\rm (1)} The pair $(\mathfrak{K},
 \mathfrak{Ind})$ is an $\mathbb{N}^I$-graded associative $\mathfrak{A}$-algebra.

{\rm (2)}
All simple perverse sheaves of weight 0 in $\mathfrak{Q}^m_{\nu}$ for various $\nu$ form an
$\mathfrak{A}$-basis of $\mathfrak{K}$ and a $\mathbb{Q}(v,t)$-basis of $\mathfrak{K}\otimes_{\mathfrak{A}}\mathbb{Q}(v,t)$.
\end{thm}
\begin{proof} (1) follows from Theorem \ref{thm1}, Proposition \ref{prop3} and the additivity of $\mathfrak{Ind}$.
(2) follows from the definition of $\mathfrak{K}$.
\end{proof}

\subsection{Defining relation}

For any $k, n\in \mathbb{N}$ and $k\leq n$, we set
$$\begin{array}{lll}
  [n]_v=\frac{v^n-v^{-n}}{v-v^{-1}},& [n]_v^!=\prod_{k=1}^n[k]_v, & \begin{bmatrix}n\\k\end{bmatrix}_v=\frac{[n]_v^!}{[k]_v^![n-k]_v^!},\\
   \begin{bmatrix}n\end{bmatrix}_{v,t}=\frac{(vt)^{n}-(vt^{-1})^{-n}}{vt-(vt^{-1})^{-1}},& [n]_{v,t}^!=\prod_{k=1}^n[k]_{v,t},& \begin{bmatrix}n\\k\end{bmatrix}_{v,t}=\frac{[n]_{v,t}^!}{[k]_{v,t}^!
  [n-k]_{v,t}^!}.
\end{array}$$
  For any $k, n\in \mathbb{N}$ and $k\leq n$, we have
  \begin{equation}\label{eq92}
\ [n]_{v,t}=t^{n-1}[n]_v,\quad \ [n]^!_{v,t}=t^{\frac{n(n-1)}{2}}[n]^!_v,\quad \ \begin{bmatrix}
      n\\k
    \end{bmatrix}_{v,t}=t^{k
    (n-k)}\begin{bmatrix}
      n\\k
    \end{bmatrix}_v.
  \end{equation}

\begin{ex}
  Let $\Omega=[1]$. The associated quiver consists of a single vertex without any arrow.
In this case, $E_V=\{{\rm pt}\}$ for any $V$.
 If $\underline{\nu}=n$, then $\mathcal{F}_{\underline{\nu}}$ is also a point.
 Then $\widetilde{L}_n={\bf 1}_{E_V}$ for $V=k^n$ where $\widetilde L_n$ is defined in Lemma ~\ref{lem1}.

If $\underline{\nu}=(n,1)$, then $\mathcal{F}_{\underline{\nu}}=\{f=(0 \subset V^1 \subset V)\ |\ \dim V^1=1\}$ is the Grassmannian $Gr(1,n+1)$.
 By Lemma 5.4.12 in \cite{BBD},
$\widetilde{L}_{(n,1)}=\oplus_{i=0}^n{\bf 1}_{E_V}[-2i](-i).$
By (\ref{eq3}), we have
 $$\mathfrak{L}_{(n,1)}= \oplus_{i=0}^n{\bf 1}_{E_V}[n-2i](n-i)=\sum_{0\leq i\leq n} v^{n-2i}t^{n}\mathfrak{L}_{n+1}=[n+1]_{v,t}\mathfrak{L}_{n+1}.$$
 In other words,
 \begin{equation}\label{eq23}
   \mathfrak{L}_n \cdot \mathfrak{L}_1=[n+1]_{v,t}\mathfrak{L}_{n+1}.
 \end{equation}
\end{ex}

\begin{ex}
Let
$\Omega=\begin{bmatrix}
  1 &-a'\\-a'' & 1
\end{bmatrix}$,
where $a', a'' \geq 0$.
The associated quiver has two vertices, say $i$ and $j$.
Let $\Omega'$ (resp. $\Omega''$) be the set of arrows from $i$ to $j$ (resp. from $j$ to $i$).
Then $a'=\#\Omega', a''=\#\Omega''$. Set $N=a'+a''$.

  Fix a vector space $V=V_i \oplus V_j$ such that $\dim V_i=1$ and $\dim V_j=N+1$. For any $p=0,1,\cdots, N+1$, let
  $$\widetilde{S}_p=\{(x,W)\in E_V\times Gr(p',V_j)\ |\ x_h(V_i)\subset W, \ {\rm if}\ h\in \Omega'; x_h|_W=0\ {\rm if}\ h\in \Omega''\},$$
  where $p'=N+1-p$ and  $Gr(p', V_j)$ is the Grassmannian of all $p'$-dimensional subspaces in $V_j$.
  Let $\pi(\widetilde{S}_p)$ be the image of the first projection map $\pi: \widetilde{S}_p\rightarrow E_V$.
  Let
  \begin{equation}\label{eq24}
    I_p'=\pi_!{\bf 1}_{\widetilde{S}_p}[\dim \widetilde{S}_p](\frac{1}{2}\dim \widetilde{S}_p).
  \end{equation}
  This is a  semisimple complex and pure  of weight 0 since $\pi$ is a proper map.
   Let $I_p:=IC(\pi(\widetilde{S}_p), {\bf 1})$ be the intersection cohomology complex of weight 0 on $E_V$ determined by $\pi(\widetilde{S}_p)$ and the constant sheaf on its smooth part.
   From  Proposition 9.4 in ~\cite{Lusztigquiver}, we have
  \begin{lem}
  \label{lem6}
  \begin{itemize}
    \item[(a)] $I_0'=I_0, \ I_{N+1}'=I_{N+1},$
    \item[(b)] $I'_p=I_p\oplus I_{p-1}\ {\rm if}\ 1 \leq p \leq a'';\quad  I'_p=I_p\oplus I_{p+1}\ {\rm if}\ a''+1 \leq p \leq N,$
    \item[(c)] $I_{a''}=I_{a''+1},$
    \item[(d)] $\dim(\widetilde{S}_p)=(p+a')(N+1-p)+a''p.$
  \end{itemize}
  \end{lem}

  %If $\underline{\nu}=pj$, then we denote by $\mathfrak{L}_{pj}$ instead of $\mathfrak{L}_{\underline{\nu}}$.
  If $\underline{\nu}=(j,j,\cdots,j)$ for $p$ iterated $j$, then we denote by $\mathfrak{L}_{j^p}$ instead of $\mathfrak{L}_{\underline{\nu}}$.
  By Proposition \ref{prop3}, we have
  $\mathfrak{L}_{j^p,i,j^{p'}}=[p]_{v,t}^![p']_{v,t}^!\mathfrak{L}_{pj,i, p'j}.$
  By (\ref{eq3}) and the definition of $I_p'$ , we further have
  $\mathfrak{L}_{j^p,i,j^{p'}}=[p]_{v,t}^![p']_{v,t}^!I_p'(\frac{1}{2}\dim(\widetilde{S}_p)).$
Thus we have
\begin{eqnarray*}
&&  \frac{1}{[N+1]_{v, t}^!}  \sum_{0\leq p\leq N+1} (-1)^pt^{-p(p'-a'+a'')}\begin{bmatrix}N+1\\p\end{bmatrix}_{v,t}
  \mathfrak{L}_{j^p,i,j^{p'}}\\
&  =&  \sum_{0\leq p\leq N+1} (-1)^pt^{-p(p'-a'+a'')}t^{(p+a')p'+a''p} I_p'
  = \sum_{0\leq p\leq N+1}  (-1)^pt^{a'(N+1)} I_p'=0,
\end{eqnarray*}
where the last equality follows from Lemma \ref{lem6}.
By combining Examples 1 and 2, we have

\begin{prop}
The following relations satisfy in  $\mathfrak{K}$ associated to any quiver $\Omega$.
\begin{eqnarray}
\label{eq101}
&\mathfrak{L}_{ni} \cdot \mathfrak{L}_{i}=[n+1]_{v,t}\mathfrak{L}_{(n+1) i }, \quad  \forall i\in I.  \\
&\sum_{p=0}^{i\cdot j+1}(-1)^pt^{-p(p'-\langle i, j\rangle +\langle j, i\rangle)}\begin{bmatrix} i\cdot j +1\\p\end{bmatrix}_{v,t}
  \mathfrak{L}_{j^p,i,j^{p'}}=0, \quad \forall i\neq  j \in I.
\end{eqnarray}
\end{prop}
\end{ex}

\subsection{Restriction functor}

Consider the following diagram
$$\xymatrix{ E_T \times E_W & F \ar[l]_-{\kappa} \ar[r]^{\iota} & E_V,}$$
where
$\iota$ is an embedding and $\kappa(x)=(x_T,x_W)$.
We define $$\widetilde{\Res}^V_{T,W}(L)=\kappa_!\iota^*L, \quad \forall L\in \mathcal{D}(E_V).$$

\begin{prop}\label{prop4}
  $\widetilde{\Res}^V_{T,W}(\widetilde{L}_{\underline{\nu}})=\oplus_{\underline{\tau}, \underline{\omega}}\widetilde{L}_{\underline{\tau}}\boxtimes \widetilde{L}_{\underline{\omega}}[-2M(\underline{\tau}, \underline{\omega})](-M(\underline{\tau}, \underline{\omega}))$, where
  \begin{equation}\label{eq6}
M(\underline{\tau}, \underline{\omega})=\sum_{h;l'<l}\tau^{l'}_{h'}\omega^l_{h''}+\sum_{i;l<l'}\tau_i^{l'}\omega_i^l
  \end{equation} and the direct sum is taken over all $\underline{\tau}$ and  $\underline{\omega}$ such that $\tau^l+ \omega^l= \nu^l$, $\sum _l \omega_i^l=\dim W_i$ and $\sum _l \tau_i^l=\dim T_i$.
\end{prop}

\begin{proof}
By 9.2.6 (b) in \cite{Lusztigbook}, we have $\widetilde{\Res}^V_{T,W}(\widetilde{L}_{\underline{\nu}})\simeq \oplus_{\underline{\tau}, \underline{\omega}}(\widetilde{L}_{\underline{\tau}}\boxtimes \widetilde{L}_{\underline{\omega}})[-2M(\underline{\tau}, \underline{\omega})]$ up to a Tate twist.
It is enough to check that the weights of the two complexes on both sides in the proposition are the same.
Let $\iota': E_T\times E_W \rightarrow F$ be the embedding map.
%The {\it hyperbolic localization} functor from $\mathcal{D}_{G_V}(E_V)$ to $\mathcal{D}_{G_T \times G_W}(E_T \times E_W)$ is defined by
%$L \mapsto (\iota')^!(\iota)^*L$.
 By (1) in \cite{braden2003hyper},
we have
$$\kappa_!\iota^*L \simeq (\iota')^!(\iota)^*L,\quad \forall L \in \mathcal{D}_{G_V}(E_V).$$
Note that the functor $(\iota')^!(\iota)^*L$ is the  hyperbolic localization functor.
By Lemma \ref{lem1} and Theorem 8 in \cite{braden2003hyper}, the weights of both complexes in the proposition are zero.
The proposition follows.
\end{proof}
 For any $L \in \mathcal{D}_{G_V}(E_V)$, we define
\begin{equation}\label{eq42}
  \mathfrak{Res}^V_{T,W}(L)=\widetilde{\Res}^V_{T,W}(L)[d_1-d_2-2\dim(G_V/P)](-\dim(G_V/P)),
\end{equation}
where $d_1$ and $d_2$ are the same as those in (\ref{eq5}).
By Theorem 8 in \cite{braden2003hyper}, Proposition \ref{prop4}, (\ref{eq3}) and (\ref{eq42}), we have the following corollary.
\begin{cor}\label{cor2}
  $\mathfrak{Res}^V_{T,W}(\mathfrak{L}_{\underline{\nu}})=\oplus_{\underline{\tau}, \underline{\omega}}\mathfrak{L}_{\underline{\tau}}\boxtimes \mathfrak{L}_{\underline{\omega}}[M'(\underline{\tau}, \underline{\omega})](M''(\underline{\tau}, \underline{\omega}))$, where
  \begin{equation}\label{eq7}
M'(\underline{\tau}, \underline{\omega})=d_1-d_2-2\dim(G_V/P)+d(\underline{\nu})-d(\underline{\tau})-d(\underline{\omega})-
2M(\underline{\tau},\underline{\omega}),
  \end{equation}
   \begin{equation}\label{eq8}
M''(\underline{\tau}, \underline{\omega})=d(\underline{\nu})-d(\underline{\tau})-d(\underline{\omega})-\dim(G_V/P)-
M(\underline{\tau},\underline{\omega}),
   \end{equation}
   and the direct sum is taken over all $\underline{\tau}$ and $\underline{\omega}$ such that $\tau^l+ \omega^l= \nu^l$, $\sum _l \omega_i^l=\dim W_i$ and $\sum _l \tau_i^l=\dim T_i$. Moreover, if $L$ is a pure complex in $\mathcal{D}_{G_V}(E_V)$, then
$${\rm wt}(\mathfrak{Res}^V_{T,W} L)={\rm wt}(L)+d_1-d_2.$$
\end{cor}

\subsection{Coalgebra structure}Define
 an $\mathfrak{A}$-linear map $\mathfrak{r}: \mathfrak{K}
\rightarrow \mathfrak{K}\otimes \mathfrak{K}$ by
 \begin{align*}
 K \mapsto
\oplus_{\tau,\omega} \mathfrak{Res}^{\nu}_{\tau, \omega}K,
 \end{align*}
for any homogenous element $K\in \mathfrak{K}_{\nu}$, where the direct sum runs through all $\tau, \omega \in \mathbb{N}^I$ such that $\tau+\omega=\nu$.

Define a multiplication on $\mathfrak{K}\otimes \mathfrak{K}$ as follows.
\begin{equation}\label{eq13}(x \otimes y)(x' \otimes y')=v^{-|x'|\cdot|y|}t^{\langle|x'|,|y|\rangle-\langle |y|,|x'|\rangle}xx' \otimes yy',
\end{equation}
for homogenous elements $x,y,x'$ and $y'$, where $|x|$ is the grading of $x$ and $\langle ,\rangle$ is defined in (\ref{eq47}).
\begin{prop}\label{prop30}
  $\mathfrak{r}: \mathfrak{K}
\rightarrow \mathfrak{K}\otimes \mathfrak{K}$ is an algebra homomorphism, where the algebra structure on $\mathfrak{K}\otimes \mathfrak{K}$ is defined in (\ref{eq13}).
\end{prop}

\begin{proof}
By Theorem \ref{thm1} and Proposition \ref{prop3}, it is enough to show that
$$\mathfrak{r}(\mathfrak{L}_{\underline{\nu}' \underline{\nu}''})=\mathfrak{r}(\mathfrak{L}_{\underline{\nu}'})\mathfrak{r}
(\mathfrak{L}_{\underline{\nu}''})\ {\rm for\ any}\ \underline{\nu}', \underline{\nu}''\in \mathbb{N}^I.$$
  By Corollary \ref{cor2},
  $$\mathfrak{r}(\mathfrak{L}_{\underline{\nu}'})=\sum \mathfrak{L}_{\underline{\tau}'} \boxtimes \mathfrak{L}_{\underline{\omega}'}[M'(\underline{\tau}',\underline{\omega}')](M''(\underline{\tau}',
  \underline{\omega}')),$$
  where the sum is taken over all $\underline{\tau}'$ and $\underline{\omega}'$ such that $\tau'^l+\omega'^l=\nu'^l$ for all $l=1,\cdots,m$.

  Similarly, we have $$\mathfrak{r}(\mathfrak{L}_{\underline{\nu}''})=\sum \mathfrak{L}_{\underline{\tau}''} \boxtimes \mathfrak{L}_{\underline{\omega}''}[M'(\underline{\tau}'',\underline{\omega}'')](M''(\underline{\tau}'',
  \underline{\omega}'')),$$
  where the sum is taken over all  $\underline{\tau}''$ and $\underline{\omega}''$ such that $\tau''^l+\omega''^l=\nu''^l$ for all $l=m+1,\cdots,m+n$. By (\ref{eq10}), we can rewrite (\ref{eq13}) as follows.
  $$(x \otimes y)(x' \otimes y')=xx' \otimes yy'[-|y|\cdot |x'|](-\langle |y|, |x'|\rangle),$$
  where $|y|\cdot |x'|=\langle|x'|, |y|\rangle+\langle |y|, |x'|\rangle$ is a symmetric bilinear form. Therefore,
  \begin{equation}\label{eq65}
\mathfrak{r}(\mathfrak{L}_{\underline{\nu}'})\mathfrak{r}(\mathfrak{L}_{\underline{\nu}''})=\sum \mathfrak{L}_{\underline{\tau}'\underline{\tau}''}\boxtimes \mathfrak{L}_{\underline{\omega}'\underline{\omega}''}[N'(\underline{\tau}'\underline{\tau}'',
  \underline{\omega}'\underline{\omega}'')]
  (N''(\underline{\tau}'\underline{\tau}'',\underline{\omega}'\underline{\omega}'')),
  \end{equation}
  where the sum is taken over all $\underline{\tau}'$, $\underline{\omega}'$, $\underline{\tau}''$ and $\underline{\omega}''$ such that $\tau'^l+\omega'^l=\nu'^l$ for all $l=1,\cdots,m$ and $\tau''^l+\omega''^l=\nu''^l$ for all $l=m+1,\cdots,m+n$. $$N'(\underline{\tau}'\underline{\tau}'',\underline{\omega}'\underline{\omega}'')
  =M'(\underline{\tau}',\underline{\omega}')+M'(\underline{\tau}'',\underline{\omega}'')-
  |\underline{\tau}''|\cdot |\underline{\omega}'|,\ {\rm and}$$
   $$N''(\underline{\tau}'\underline{\tau}'',\underline{\omega}'\underline{\omega}'')
  =M''(\underline{\tau}',\underline{\omega}')+M''(\underline{\tau}'',\underline{\omega}'')-\langle |\underline{\omega}'|, |\underline{\tau}''|\rangle.$$
  On the other hand, we have
  \begin{equation}\label{eq66}
\mathfrak{r}(\mathfrak{L}_{\underline{\nu}'\underline{\nu}''})=\sum \mathfrak{L}_{\underline{\tau}'\underline{\tau}''} \boxtimes \mathfrak{L}_{\underline{\omega}'\underline{\omega}''}[M'(\underline{\tau}'\underline{\tau}'',
  \underline{\omega}'\underline{\omega}'')]
  (M''(\underline{\tau}'\underline{\tau}'',\underline{\omega}'\underline{\omega}'')).
  \end{equation}
 By comparing (\ref{eq65}) with (\ref{eq66}), it remains to show that
  \begin{equation}\label{eq16}
 M'(\underline{\tau}'\underline{\tau}'',\underline{\omega}'\underline{\omega}'')
  =N'(\underline{\tau}'\underline{\tau}'',\underline{\omega}'\underline{\omega}''),\ {\rm and}\  M''(\underline{\tau}'\underline{\tau}'',\underline{\omega}'\underline{\omega}'')=
  N''(\underline{\tau}'\underline{\tau}'',\underline{\omega}'\underline{\omega}'').
  \end{equation}
  The proof of the first one is the same as that of Lemma 13.1.5 in \cite{Lusztigbook}. We only need to show that the second one holds.

  By equations (\ref{eq7}) and (\ref{eq8}), we have
 \begin{eqnarray}\label{eq17}
   & &\ \  M''(\underline{\tau}'\underline{\tau}'',\underline{\omega}'\underline{\omega}'')
   -N''(\underline{\tau}'\underline{\tau}'',\underline{\omega}'\underline{\omega}'') \\
    & & = d(\underline{\nu}'\underline{\nu}'')-d(\underline{\tau}'\underline{\tau}'')-
    d(\underline{\omega}'\underline{\omega}'')
    -\dim(G_V/P)_{\underline{\tau}'\underline{\tau}'',\underline{\omega}'\underline{\omega}''}-
    M(\underline{\tau}'\underline{\tau}'',\underline{\omega}'\underline{\omega}'') \nonumber\\
    & & -(d(\underline{\nu}')-d(\underline{\tau}')-d(\underline{\omega}')-
    \dim(G_V/P)_{\underline{\tau}',\underline{\omega}'}
    -M(\underline{\tau}',\underline{\omega}')) \nonumber\\
   & & -(d(\underline{\nu}'')-d(\underline{\tau}'')-d(\underline{\omega}'')-
   \dim(G_V/P)_{\underline{\tau}'',\underline{\omega}''}
   -M(\underline{\tau}'',\underline{\omega}''))+\langle|\underline{\omega}'|,|\underline{\tau}''| \rangle, \nonumber
  \end{eqnarray}
  where $\dim(G_V/P)_{\underline{\tau},\underline{\omega}}=\sum_i\underline{\tau}_i\underline{\omega}_i$ and $\underline{\tau}_i=\sum_l \tau^l_i.$ Moreover,
  \begin{eqnarray}\label{eq20}
    & &\dim(G_V/P)_{\underline{\tau}'\underline{\tau}'',\underline{\omega}'\underline{\omega}''}
    -\dim(G_V/P)_{\underline{\tau}',\underline{\omega}'}-
    \dim(G_V/P)_{\underline{\tau}'',\underline{\omega}''}
     = \sum_i\tau'_i\omega''_i+\tau''_i\omega'_i.
  \end{eqnarray}
  In general, if $\underline{\nu}=\underline{\tau}+\underline{\omega}$, we have
  $$d(\underline{\nu})-d(\underline{\tau})-d(\underline{\omega})=\sum_{h;l'<l}
  \tau^{l'}_{h'}\omega^l_{h''}+\omega^{l'}_{h'}\tau^l_{h''}+
  \sum_{i;l<l'}\tau_i^{l'}\omega_i^l+\omega_i^{l'}\tau_i^l.$$
  Hence,
  \begin{eqnarray}\label{eq18}
  & &d(\underline{\nu}'\underline{\nu}'')-d(\underline{\tau}'\underline{\tau}'')-
  d(\underline{\omega}'\underline{\omega}'')-(d(\underline{\nu}')-
  d(\underline{\tau}')-d(\underline{\omega}'))\\
  & &\ \ -(d(\underline{\nu}'')-d(\underline{\tau}'')-d(\underline{\omega}''))\nonumber
   =\sum_h \tau'_{h'}\omega''_{h''}+\omega'_{h'}\tau''_{h''}+\sum_i\tau''_i\omega'_i+\omega''_i\tau'_i, \nonumber
  \end{eqnarray}
  where $\tau'_{h'}=\sum_l \tau'^l_{h'}$.
  By (\ref{eq6}), we have
  \begin{eqnarray}\label{eq19}
    M(\underline{\tau}'\underline{\tau}'',\underline{\omega}'\underline{\omega}'')-
   M(\underline{\tau}',\underline{\omega}')-
   M(\underline{\tau}'',\underline{\omega}'')
    = \sum_h \tau'_{h'}\omega''_{h''}+\sum_i\tau''_i\omega_i'.
  \end{eqnarray}

  By (\ref{eq20}), (\ref{eq18}) and (\ref{eq19}), the right hand side of (\ref{eq17}) is 0. This shows that (\ref{eq16}) holds.
\end{proof}

\subsection{Bar involution}

Given any pure complex $K \in \mathfrak{Q}_V^m$, let
$$\mathfrak{D}(K)=(\mathbb{D}K)(-{\rm wt}(K)).$$
 Since objects in $\mathfrak{Q}_V^m$ are semisimple, this defines a functor $\mathfrak{D}: \mathfrak{Q}_V^m \rightarrow \mathfrak{Q}_V^m$.
We notice that $\mathfrak{D}^2$ is the identity functor. By (\ref{eq3}) and Section \ref{sec3.1}(d), $\mathbb{D}\mathfrak{L}_{\nu}=\mathfrak{L}_{\nu}(-d(\nu))$ and ${\rm wt}(\mathfrak{L}_{\nu})=-d(\nu)$. So we have
\begin{equation}\label{eq103}
  \mathfrak{D}(\mathfrak{L}_{\nu})=\mathfrak{L}_{\nu}.
\end{equation}

\begin{prop}\label{prop12}
If both $K \in \mathfrak{Q}^m_T$ and $L \in \mathfrak{Q}^m_W$ are pure, then we have
$$\mathfrak{D}\circ \mathfrak{Ind} (K\boxtimes L)=\mathfrak{Ind}\circ \mathfrak{D} (K\boxtimes L).$$
\end{prop}
\begin{proof}
  By Proposition 9.2.5 in \cite{Lusztigbook}, it is enough to check that the weights on both sides equal. The proposition follows from the fact that $\mathfrak{D}$ preserves the weights of pure complexes.
\end{proof}
Define an involution $^-{}\!: \mathfrak{K} \rightarrow \mathfrak{K}$ by
\begin{eqnarray}\label{eq67}
v \mapsto v^{-1}, \ t \mapsto t,\ {\rm and}\ K\mapsto \mathfrak{D}(K).
\end{eqnarray}
\begin{prop}\label{prop13}
The map $^-{}\!: \mathfrak{K} \rightarrow \mathfrak{K}$ is a $\mathbb{Q}(t)$-algebra involution.
\end{prop}
\begin{proof}
  By Proposition \ref{prop3}, Theorem \ref{thm1} and (\ref{eq103}), it is enough to check that
  $$\overline{v \cdot K}=v^{-1}\cdot \overline{K}, \ {\rm and}\ \overline{t \cdot K}=t \cdot \overline{K}, \ \forall K\in \mathfrak{K}.$$
  By the definition of $``^-"$, we can assume that $K$ is the isomorphism class of a pure complex. Set ${\rm wt}(K)=w$. Then we have
  $$\overline{v \cdot K}=\overline{K[1](\frac{1}{2})}=(\mathbb{D}(K[1](\frac{1}{2})))(-w)=(\mathbb{D}(K))[-1](-w-\frac{1}{2}).$$
  On the other hand,
  $$v^{-1}\cdot \overline{K}=v^{-1}\cdot(\mathbb{D}K)(-w)=(\mathbb{D}(K))[-1](-w-\frac{1}{2}).$$
  Similarly, one can check that $\overline{t \cdot K}=t\cdot \overline{K}$.
\end{proof}

\subsection{Bilinear form}\label{sec5.4}

Recall that for any two $G$-equivariant semisimple complexes $K$, $L$ on $X$, one can define a number, $d_j(K,L) \in \mathbb{Z}_{\geq 0}$, for any $j \in \mathbb{Z}$ which Lusztig denotes by  $d_j(X, G; K,L)$ (\cite{GL,Lusztigcharsheaf1,Lusztigbook}).
Given any two pure complexes $K,L \in \mathfrak{Q}^m_V$, we define
\begin{equation}\label{eq41}
  (K,L)=\sum_{j\in \mathbb{Z}}d_j(K,L)v^{-j}t^{-({\rm wt}(K)+{\rm wt}(L))}.
\end{equation}
Since each element in $\mathfrak{Q}^m_V$ is semisimple and simple perverse sheaves are pure, we can extend this definition to entire $\mathfrak{Q}^m_V$. This induces a bilinear form on $\mathfrak{K}$.

Given any pure complexes $K_1, K_2 \in \mathfrak{Q}^m_T$ and $L_1, L_2 \in \mathfrak{Q}^m_W$, we define
$$(K_1 \boxtimes L_1, K_2 \boxtimes L_2)=t^{2d}(K_1, K_2)\cdot (L_1, L_2),$$
where $d=\sum_h\dim T_{h'}\dim W_{h''}+\sum_i\dim T_i\dim W_i.$
Similarly, this can be extended linearly to a bilinear form on $\mathfrak{K} \otimes \mathfrak{K}.$

\begin{prop}\label{prop9}
  Let $K \in \mathfrak{Q}^m_T, L \in \mathfrak{Q}^m_W$ and $M \in \mathfrak{Q}^m_V$ such that $V=T\oplus W$, then
  $$(\mathfrak{Ind}^V_{T,W}(K\boxtimes L), M)=(K\boxtimes L, \mathfrak{Res}^V_{T,W}M).$$
\end{prop}
\begin{proof}Since both $\mathfrak{Ind}$ and $\mathfrak{Res}$ are additive functors, we can assume that $K,L$ and $M$ are pure complexes. By Lemma 7 in \cite{GL}, we have
\begin{eqnarray*}
   &(\mathfrak{Ind}^V_{T,W}(K\boxtimes L), M)
  = \sum_{j\in \mathbb{Z}} d_j(\mathfrak{Ind}^V_{T,W}(K\boxtimes L), M)v^{-j}t^{-({\rm wt}(\mathfrak{Ind}^V_{T,W}(K\boxtimes L))+{\rm wt}(M))}\hspace{40pt}\\
  &= \sum_{j\in \mathbb{Z}} d_j(K\boxtimes L, \mathfrak{Res}^V_{T,W}M)v^{-j}t^{-({\rm wt}(K)+{\rm wt}(L)+{\rm wt}(M))+d}\hspace{180pt}\\
  &= \sum_{j\in \mathbb{Z}} d_j(K\boxtimes L, \oplus M_1 \boxtimes M_2)v^{-j}t^{-({\rm wt}(K)+{\rm wt}(L)+{\rm wt}(M_1)+{\rm wt}(M_2))+2d}
  %&=& \sum \sum_j \sum_i d_i(E_T, G_T;K, M_1) d_{j-i}(E_W, G_W; L, M_2)(\alpha (vt)^{--1})^{-\frac{1}{2}({\rm wt}(K)+{\rm wt}(L)+{\rm wt}(M_1)+{\rm wt}(M_2))+d}\ {\it (by\ K\ddot{u}nneth\ formula)}\\
  = \sum (K, M_1)\cdot (L, M_2)t^{2d}.
\end{eqnarray*}
The last equality follows from 8.1.10 (f) in \cite{Lusztigbook}.
On the other hand,
$$(K\boxtimes L, \mathfrak{Res}^V_{T,W}M)=(K\boxtimes L, \oplus M_1\boxtimes M_2)=\sum (K, M_1)\cdot (L, M_2)t^{2d}.$$
So $(\mathfrak{Ind}^V_{T,W}(K\boxtimes L), M)=(K\boxtimes L, \mathfrak{Res}^V_{T,W}M).$
\end{proof}

\subsection{Fourier-Deligne transformation}

Let ${}'\!\Omega$ be a second quiver such that the underlying graph is the same as that of $\Omega$.
Denote the source of
the arrow $h$ in ${}'\!\Omega$ by $s(h)={}'\!h$ and its target by $t(h)={}''\!h$. Recall that the source and the target of
the arrow $h$ in $\Omega$ are denoted by $s(h)=h'$ and $t(h)=h''$, respectively. Let $\Omega_1=
\left\{ h \in \Omega \mid {}'\!h=h', {}''\!h=h'' \right\}$ and $\Omega_2= \left\{
h \in \Omega \mid {}'\!h=h'',{}''\!h=h' \right\}$. For a given $I$-graded
$k$-vector space $V$, we denote
\begin{equation*}
\begin{split}
&{}'\!E_V=\bigoplus_{h \in \Omega_1} \Hom(V_{h'},V_{h''}) \oplus \bigoplus_{h \in
\Omega_2} \Hom(V_{h''},V_{h'}),\\
&\dot{E}_V=\bigoplus_{h \in \Omega_1} \Hom(V_{h'},V_{h''}) \oplus
\bigoplus_{h \in \Omega_2} \Hom(V_{h'},V_{h''}) \oplus \bigoplus_{h \in \Omega_2}
\Hom(V_{h''},V_{h'}).
\end{split}
\end{equation*}
 We have the natural projection maps
$$\xymatrix{E_V & \dot{E}_V \ar[l]_-{s}\ar[r]^-{t}& {}'\!E_V.}$$
Recall that to a nontrivial character, $\varphi$, of $\mathbb{F}_p$, one can associate a local system $\mathcal{L}_{\varphi}$ on $k$ of rank one.
Let $\mathcal{T}_V: \dot{E}_V  \rightarrow k $ be the map defined by
\begin{equation}\label{eq9}
\mathcal{T}_V(a,b,c)=\sum_{h \in \Omega_2}  Tr(V_{h'}
\xrightarrow{b} V_{h''} \xrightarrow{c} V_{h'}),
\end{equation}
 where $Tr$ is the
trace function.
Denote $\mathcal{L}_{\mathcal{T}_V}=\mathcal{T}_V^*\mathcal{L}_{\varphi}$ which
is a rank one $\overline{\mathbb{Q}}_l$-local system on $\dot{E}_V$.
The {\it Fourier-Deligne transformation} $\Phi: \mathcal{D}(E_V) \rightarrow
\mathcal{D}({}'\!E_V)$ is defined by
$$ L \mapsto t_!(s^*(L) \otimes
\mathcal{L}_{\mathcal{T}_V}[d_V](\frac{1}{2}d_V)),$$
 where $d_V={\rm dim} (\oplus_{h \in \Omega_2}
\Hom(V_{h'}, V_{h''}))$.

For the quiver ${}'\!\Omega$, one can similarly define ${}'\!\widetilde{L}_{\underline{\nu}}$, (resp. ${}'\!\mathfrak{L}_{\underline{\nu}}$ and ${}'\!\mathfrak{Q}^m_V$) as we define $\widetilde{L}_{\underline{\nu}}$, (resp. $\mathfrak{L}_{\underline{\nu}}$ and $\mathfrak{Q}^m_V$) for the quiver $\Omega$.

\begin{prop}\label{prop6}
 {\rm (a)} The
Fourier-Deligne transformation $\Phi$ preserves purity and weight.

 {\rm (b)} $\Phi(\widetilde{L}_{\underline{\nu}})={}'\!\widetilde{L}_{\underline{\nu}}[M](\frac{M}{2}),$ where
 $M=\sum_{h \in \Omega_2; l'<l} \nu^l_{h'}\nu^{l'}_{h''}-\nu^{l'}_{h'}\nu^{l}_{h''}.$

 {\rm (c)} $\Phi(\mathfrak{L}_{\underline{\nu}})={}'\!\mathfrak{L}_{\underline{\nu}}(-\frac{M}{2}).$
\end{prop}
\begin{proof}
  By the definition of $\mathcal{L}_{\mathcal{T}_V}$ and Section \ref{sec3.1} (h), we have ${\rm wt}(\mathcal{L}_{\mathcal{T}_V})=0$. Set ${\rm wt}(L)=w$. By Section \ref{sec3.1} (h) again, we have
${\rm wt}(s^*(L) \otimes
\mathcal{L}_{\mathcal{T}_V})={\rm wt}(L)=w,$
i.e., $s^*(L) \otimes
\mathcal{L}_{\mathcal{T}_V} \in \mathcal{D}_{\leq w}({\dot E}_V)\bigcap \mathcal{D}_{\geq w}({\dot E}_V)$.
So
\begin{equation}\label{eq76}
t_!(s^*(L) \otimes
\mathcal{L}_{\mathcal{T}_V}) \in \mathcal{D}_{\leq w}({\dot E}_V).
\end{equation}
On the other hand, $t_*(s^*(L) \otimes
\mathcal{L}_{\mathcal{T}_V}) \in \mathcal{D}_{\geq w}({\dot E}_V).$
By ~\cite[2.1.3]{hotta1984},
$$t_!(s^*(L) \otimes
\mathcal{L}_{\mathcal{T}_V}) \simeq  t_*(s^*(L) \otimes
\mathcal{L}_{\mathcal{T}_V}),\quad \forall L\in \mathcal{D}(E_V).$$
Therefore, we have
\begin{equation}\label{eq77}
t_!(s^*(L) \otimes
\mathcal{L}_{\mathcal{T}_V}) \in \mathcal{D}_{\geq w}({\dot E}_V).
\end{equation}
By (\ref{eq76}) and (\ref{eq77}), we have ${\rm wt}(t_!(s^*(L) \otimes
\mathcal{L}_{\mathcal{T}_V}))=w={\rm wt}(L)$. Part (a) follows.

 By Proposition 10.2.2 in \cite{Lusztigbook}, we have
$\Phi(\widetilde{L}_{\underline{\nu}})\simeq {}'\!\widetilde{L}_{\underline{\nu}}[M]$ up to a Tate twist.
So it is enough to check the weights on both sides equal. By Lemma \ref{lem1} and part (a),
${\rm wt}(\Phi(\widetilde{L}_{\underline{\nu}})={\rm wt}({}'\!\widetilde{L}_{\underline{\nu}}[M](\frac{M}{2}))=0.$ Part (b) follows.

 By part (b) and (\ref{eq3}), part (c) follows from the fact that
${}'\!d(\underline{\nu})=M+d(\underline{\nu}).$
\end{proof}
Similarly,  one can define a functor
$\Phi: \mathcal{D}(E_T \times E_W) \rightarrow \mathcal{D}({}'\!E_T \times {}'\!E_W)$ by
replacing $E_V$ (resp. ${}'\!E_V$) by $E_T\times E_W$ (resp. ${}'\!E_T \times {}'\!E_W$).

\begin{prop}\label{prop7}
For any $K \in \mathfrak{Q}^m_T$ and $L\in \mathfrak{Q}^m_W$, we have
  $$\Phi(\mathfrak{Ind}^V_{T,W}(K \boxtimes L))=\mathfrak{Ind}^V_{T,W}(\Phi(K\boxtimes L))(-d_0),$$ where
  $d_0=\frac{1}{2}\sum_{h \in \Omega_2}(\dim T_{h'}\dim W_{h''}-\dim W_{h'}\dim T_{h''}).$
\end{prop}
\begin{proof} Since both $\Phi$ and $\mathfrak{Ind}^V_{T,W}$ are additive functors, we can assume that $K$ and $L$ are pure complexes.
By Proposition 10.2.6 in \cite{Lusztigbook}, we have
$\Phi(\mathfrak{Ind}^V_{T,W}(K \boxtimes L))\simeq \mathfrak{Ind}^V_{T,W}(\Phi(K\boxtimes L))$ up to a Tate twist.
So it is enough to check that weights on both sides equal. By Proposition \ref{prop2} and Proposition \ref{prop6},
$${\rm wt}(\Phi(\mathfrak{Ind}^V_{T,W}(K \boxtimes L)))={\rm wt}(K)+{\rm wt}(L)-(d_1-d_2),\ {\rm and}$$
 $${\rm wt}(\mathfrak{Ind}^V_{T,W}(\Phi(K\boxtimes L))(-d_0))={\rm wt}(K)+{\rm wt}(L)-({}'\!d_1-{}'\!d_2)+2d_0,$$
where ${}'\!d_1$ (resp. ${}'\!d_2$) is defined similarly as $d_1$ (resp. $d_2$) for the new orientation. The proposition follows from the fact that
${}'\!d_1-{}'\!d_2-(d_1-d_2)=2d_0.$
\end{proof}
\begin{prop}\label{prop8} For any $K \in \mathfrak{Q}^m_T$ and $L\in \mathfrak{Q}^m_W$, we have
  $$\Phi(\mathfrak{Res}^V_{T,W}(K))=\mathfrak{Res}^V_{T,W}(\Phi(K))(d_0).$$
\end{prop}
\begin{proof} This proposition can be proved similarly as
 Proposition \ref{prop7}.
\end{proof}

From Proposition \ref{prop7}, the algebra structure of $\mathfrak{K}$ depends on the orientation of the quiver. We shall show that $\mathfrak{K}$ is independent of the orientation under a twisted multiplication.
We define
\begin{eqnarray}\label{eq21}
\widehat{\mathfrak{Ind}}^V_{T,W}(K \boxtimes L)=\mathfrak{Ind}^V_{T,W}(K \boxtimes L)(-\frac{d}{2}) \ {\rm and}\quad
\widehat{\mathfrak{Res}}^V_{T,W}(K)=\mathfrak{Res}^V_{T,W}(K)(\frac{d}{2}), %\label{eq53}
\end{eqnarray}
where $d$ is the same as the one in Section \ref{sec5.4}. By Propositions \ref{prop7} and \ref{prop8}, the following proposition holds.
\begin{prop}\label{prop10}
  For any $K \in \mathfrak{Q}^m_T$ and $L\in \mathfrak{Q}^m_W$, we have
  \begin{eqnarray*}
\Phi(\widehat{\mathfrak{Ind}}^V_{T,W}(K \boxtimes L))=\widehat{\mathfrak{Ind}}^V_{T,W}(\Phi(K\boxtimes L)) \ {\rm and}\quad
  \Phi(\widehat{\mathfrak{Res}}^V_{T,W}(K))=\widehat{\mathfrak{Res}}^V_{T,W}(\Phi(K)).
  \end{eqnarray*}
\end{prop}

\section{The algebra $\mathfrak{f}$}\label{sec2}

\subsection{The free algebra ${}'\! \mathfrak{f}$}\label{sec2.01}

Recall that $\Omega$ is a matrix satisfying (a),(b),(c) in Section \ref{sec2.1}. In this section, we drop the assumption that $\Omega_{ii} =1$ for any $i\in I$.

For indeterminates $v$ and $t$,
 we set $v_i=v^{i\cdot i/2}$ and $t_i=t^{i\cdot i/2}$.
Moreover, for any rational function $P \in \mathbb{Q}(v,t)$, let $P_i$ stand for the rational function obtained from $P$ by substituting $v, t$ by $v_i, t_i$, respectively.
  We set $v_{\nu}=\prod_i v_i^{\nu_i}$ and ${\rm tr}(\nu) =\sum_{i\in I} \nu_i \in \mathbb{Z}$, for any $\nu=(\nu_i)_{i\in I}\in \mathbb{Z}^I$. $t_{\nu}$ is defined similarly.

Let ${}'\! \mathfrak{f}$ be the free unital associative algebra over $\mathbb{Q}(v,t)$ generated by the symbols $ \theta_i,\ \forall i\in I$. By setting the degree of the generator $\theta_i$ to be $i$, the algebra ${}'\! \mathfrak{f}$ becomes an $\mathbb{N}^I$-graded algebra. We denote by ${}'\!\mathfrak{f}_{\nu}$ the subspace of all homogenous elements of degree $\nu$. We have ${}'\! \mathfrak{f}=\oplus_{\nu\in \mathbb{N}^I}{}'\! \mathfrak{f}_{\nu}$, and we denote by $|x|$ the degree of a homogenous element $x\in {}'\! \mathfrak{f}$.

On the tensor product ${}'\!\mathfrak{f}\otimes {}'\!\mathfrak{f}$, we define an associative $\mathbb{Q}(v,t)$-algebra structure by
\begin{equation}\label{eq25}(x_1 \otimes x_2)(y_1 \otimes y_2)=v^{-|y_1|\cdot|x_2|}t^{\langle|y_1|,|x_2|\rangle-\langle |x_2|,|y_1|\rangle}x_1y_1 \otimes x_2y_2,
\end{equation}
for homogeneous elements $x_1,x_2,y_1$ and $y_2$ in ${}'\!\mathfrak{f}$.
It is associative since the forms $\langle , \rangle$ in (\ref{eq47}) and $``\cdot "$ in (\ref{eq49}) are bilinear.

Similarly, on ${}'\!\mathfrak{f}\otimes {}'\!\mathfrak{f}\otimes {}'\!\mathfrak{f}$, we define an associative $\mathbb{Q}(v,t)$-algebra structure by
\begin{eqnarray}\label{eq26}
& &(x_1 \otimes x_2\otimes x_3)(y_1 \otimes y_2\otimes y_3)
 =v^{-M}t^{N}
x_1y_1 \otimes x_2y_2\otimes x_3y_3,
\end{eqnarray}
for any homogeneous elements $x_1,x_2,x_3,y_1,y_2$ and $y_3$, where $$M=|x_3|\cdot|y_1|+|x_2|\cdot|y_1|+|x_3|\cdot|y_2| \quad  {\rm and}$$
$$N=\langle
|y_1|,|x_3|\rangle+\langle|y_1|,|x_2|\rangle+\langle|y_2|,|x_3|\rangle-
\langle|x_3|,|y_1|\rangle-\langle|x_2|,|y_1|\rangle-\langle|x_3|,|y_2|\rangle.$$

By the equations (\ref{eq25}) and (\ref{eq26}), one can check that
\begin{eqnarray}\label{eq27}
  & &(x_1 \otimes x_2\otimes x_3)(y_1 \otimes y_2\otimes y_3)\\
  &=&v^{-|x_3|\cdot (|y_1|+|y_2|)}t^{\langle|y_1|+|y_2|,|x_3|\rangle-\langle|x_3|, |y_1|+|y_2|\rangle}((x_1 \otimes x_2)(y_1 \otimes y_2))\otimes x_3y_3. \nonumber
\end{eqnarray}

Let $r: {}'\!\mathfrak{f}\rightarrow  {}'\!\mathfrak{f}\otimes {}'\!\mathfrak{f}$ be the $\mathbb{Q}(v,t)$-algebra homomorphism such that
$$r(\theta_i)=\theta_i\otimes 1+1 \otimes \theta_i,\quad {\rm for\ all}\ i\in I.$$
\begin{lem}\label{lem7}
 The linear maps  $(r\otimes 1)r$, $(1\otimes r)r: {}'\!\mathfrak{f}\rightarrow  {}'\!\mathfrak{f}\otimes {}'\!\mathfrak{f} \otimes  {}'\!\mathfrak{f}$ are algebra homomorphisms. Moreover, we have the coassociativity property
  $(r\otimes 1)r=(1\otimes r)r.$
\end{lem}
\begin{proof}
 By the equation (\ref{eq27}) and the bilinearity of $\langle , \rangle$,  $r\otimes 1$ is an algebra homomorphism. Similarly,  $ 1\otimes r$ is an algebra homomorphism. The first statement follows. The second statement follows from the fact that $$(r\otimes 1)r(\theta_i)=\theta_i\otimes 1\otimes 1+1\otimes \theta_i \otimes 1+1\otimes 1\otimes \theta_i=(1\otimes r)r(\theta_i),$$
 for all $i \in I$.
\end{proof}

\begin{prop}\label{prop14}
  There is a unique bilinear form {\rm (,)} on ${}'\!\mathfrak{f}$ with values in $\mathbb{Q}(v,t)$ such that
  \begin{itemize}
    \item[(a)] $(1,1)=1$;
    \item[(b)] $(\theta_i, \theta_j)=\delta_{ij}\frac{1}{1-v^{-2}_i}$,\quad for all $i, j \in I$;
    \item[(c)] $(x, y'y'')=(r(x), y' \otimes y'')$,\quad for all $x, y', y'' \in {}'\!\mathfrak{f}$;
    \item[(d)] $(x'x'', y)=(x'\otimes x'', r(y))$,\quad for all $x', x'', y \in {}'\!\mathfrak{f}$.
  \end{itemize}
  Here the bilinear form on ${}'\!\mathfrak{f} \otimes {}'\!\mathfrak{f}$ is defined by
  \begin{equation}\label{eq74}
    (x_1 \otimes x_2, y_1 \otimes y_2)=t^{2[|x_1|,|x_2|]}(x_1, y_1) (x_2, y_2),
  \end{equation}
   where {\rm [,]} is defined in (\ref{eq48}). Moreover, the bilinear form on $ {}'\!\mathfrak{f}$ is symmetric.
\end{prop}
\begin{proof} The proof goes in a similar way as that of Proposition 1.2.3 in \cite{Lusztigbook}. For the convenience of the reader, we present it here.
  Let $ {}'\!\mathfrak{f}^*_{\nu}$ be the dual space of $ {}'\!\mathfrak{f}_{\nu}$. We define a bilinear map by
  \begin{eqnarray}\label{eq28}
  \star\ :\   {}'\!\mathfrak{f}^*_{\nu} \times {}'\!\mathfrak{f}^*_{\nu'} \rightarrow {}'\!\mathfrak{f}^*_{\nu+\nu'},\quad
    f,g \mapsto  f\star g:=(f\otimes g)r.
  \end{eqnarray}
  By Lemma \ref{lem7}, we have
  $$((f\star g)\star h)=(f\otimes g \otimes h)(r\otimes 1)r=(f\otimes g \otimes h)(1\otimes r)r=f\star (g\star h)).$$
  So the bilinear map $\star$  defines an associative algebra structure on $\oplus_{\nu \in\mathbb{N}^I} {}'\!\mathfrak{f}^*_{\nu}.$ Now we define a new multiplication on $\oplus_{\nu} {}'\!\mathfrak{f}^*_{\nu}$ by
   \begin{eqnarray}\label{eq29}
  \circ \ :\   {}'\!\mathfrak{f}^*_{\nu} \times {}'\!\mathfrak{f}^*_{\nu'} \rightarrow  {}'\!\mathfrak{f}^*_{\nu+\nu'},\quad
    f,g\mapsto  f\circ g:=t^{2[\nu, \nu']} f\star g.
  \end{eqnarray}
  Since $[, ]$ is a bilinear form, $\oplus_{\nu} {}'\!\mathfrak{f}^*_{\nu}$ equipped with $``\circ"$ is also an associative algebra. For the rest of the proof, we assume that $\oplus_{\nu} {}'\!\mathfrak{f}^*_{\nu}$ is the algebra equipped with the multiplication $``\circ"$.

  For any $i\in I$, let $\vartheta_i \in  {}'\!\mathfrak{f}^*_{i}$ be the linear map given by
  $\vartheta_i(\theta_i)=\frac{1}{1-v^{-2}_i}.$ Since ${}'\!\mathfrak{f}$ is a free algebra, there is a unique algebra homomorphism $\zeta: {}'\!\mathfrak{f} \rightarrow \oplus_{\nu} {}'\!\mathfrak{f}^*_{\nu}$ such that $\zeta(\theta_i)=\vartheta_i$ for all $i \in I$.
  For any $x, y \in {}'\!\mathfrak{f}$, we set
  \begin{equation}\label{eq50}
(x,y)=\zeta(y)(x).
  \end{equation}
  By the definition of $\zeta$, (a) and (b) in the proposition follows automaically.
  We now show that (c) holds. Since $\zeta$ is an $\mathbb{N}^I$-graded algebra homomorphism and preserves the grading, we have
  \begin{equation}\label{eq30}
    (x, y)=0\ {\rm if}\ x,y\ {\rm are\ homogeneous\ and}\ |x| \neq |y|.
  \end{equation}
  We write $r(x)=\sum x_1 \otimes x_2$.
  By (\ref{eq50}), we have the following.
  \begin{eqnarray*}
  &  &(x, y'y'')=\zeta(y'y'')(x)
   =t^{2[|y'|,|y''|]}(\zeta(y')\star \zeta(y''))(x)=t^{2[|y'|,|y''|]}(\zeta(y')\otimes \zeta(y''))r(x)\\
    & &=t^{2[|y'|,|y''|]}\sum(\zeta(y')\otimes \zeta(y''))(x_1 \otimes x_2)=t^{2[|y'|,|y''|]}\sum(x_1,y')(x_2,y'')
    =(r(x), y' \otimes y'').
  \end{eqnarray*}
  Hence, (c) follows.

 Next, we show that (d) holds.
Suppose that $y=\theta_i$ for some $i\in I$.
If $x'=\theta_i$ and $x''=1$ or $x'=1$ and $x''=\theta_i$, we have
  $$(x'x'', y)=(1-v_i^{-2})^{-1}=(x'\otimes x'', r(y)).$$
  By (\ref{eq30}), (d) holds for the case $y=\theta_i$. Now we assume that (d) holds for $y'$ and $y''$, we are going to show that (d) holds for $y=y'y''$. Due to the fact that (,) is a bilinear form, we can assume that $x', x'', y', y''$ are all homogeneous.
  Let
  \begin{eqnarray*}
    r(x')=\sum x'_1\otimes x'_2,\quad  r(x'')=\sum x''_1\otimes x''_2,\quad
    r(y')=\sum y'_1\otimes y'_2,\quad r(y'')=\sum y''_1\otimes y''_2,
  \end{eqnarray*}
  such that all factors are homogeneous. Then
  $$r(x'x'')=\sum v^{-|x'_2|\cdot|x''_1|}t^{\langle|x''_1|,|x'_2|\rangle-\langle |x'_2|,|x''_1|\rangle}(x'_1x''_1 \otimes x'_2x''_2),\ {\rm and}$$
  $$r(y'y'')=\sum v^{-|y'_2|\cdot|y''_1|}t^{\langle|y''_1|,|y'_2|\rangle-\langle |y'_2|,|y''_1|\rangle}(y'_1y''_1 \otimes y'_2y''_2).$$
   So we have
   \begin{equation*}
   \begin{split}
     (&x'x'',y'y'')=\zeta(y'y'')(x'x'')
   = t^{2[|y'|,|y''|]} (\zeta(y')\otimes \zeta(y''))r(x'x'')\\
   &= \sum v^{-|x'_2|\cdot|x''_1|}t^{C_1}
   (x'_1x''_1,y')(x'_2x''_2, y'')
   = \sum v^{-|x'_2|\cdot|x''_1|}t^{C_1}
   (x'_1\otimes x''_1,r(y'))(x'_2\otimes x''_2, r(y''))\\
   &= \sum v^{-|x'_2|\cdot|x''_1|}t^{C}(x'_1, y'_1)(x''_1,y'_2)(x'_2, y''_1)(x''_2, y''_2),
   \end{split}
  \end{equation*}
  where $C_1=\langle|x''_1|,|x'_2|\rangle-\langle |x'_2|,|x''_1|\rangle+2[|y'|,|y''|]$ and $C=\langle|x''_1|,|x'_2|\rangle-\langle |x'_2|,|x''_1|\rangle+2([|y'|,|y''|]+
   [|x'_1|,|x''_1|]+[|x'_2|,|x''_2|])$.
  On the other hand, $(x'\otimes x'', r(y'y''))$ is equal to

 \begin{equation*}
 \begin{split}
    \sum& v^{- |y'_2|\cdot|y''_1|}t^{\langle|y''_1|,|y'_2|\rangle-\langle |y'_2|,|y''_1|\rangle}
   (x'\otimes x'', y'_1y''_1\otimes  y'_2y''_2)
   = \sum v^{- |y'_2|\cdot|y''_1|}t^{D_1}
   (x', y'_1y''_1)(x'', y'_2y''_2)\\
   =\sum& v^{- |y'_2|\cdot|y''_1|}t^{D_1}
   (r(x'), y'_1\otimes y''_1)(r(x''), y'_2\otimes y''_2)
   = \sum v^{- |y'_2|\cdot|y''_1|}t^{D}
(x'_1, y'_1)(x''_1,y'_2)(x'_2, y''_1)(x''_2, y''_2),
 \end{split}
  \end{equation*}
  where $D_1=\langle|y''_1|,|y'_2|\rangle-\langle |y'_2|,|y''_1|\rangle+2[|x'|,|x''|]$ and $D=\langle|y''_1|,|y'_2|\rangle-\langle |y'_2|,|y''_1|\rangle+2([|x'|,|x''|]+[|x'_1|,
   |x'_2| ]+[|x''_1|, |x''_2|])$.
  By (\ref{eq30}) and induction hypothesis, Part (d) is reduced to show that $C=D$ under the following assumption.
  $$|x'_1|=|y'_1|, |x'_2|=|y''_1|, |x''_1|=|y'_2|, |x''_2|=|y''_2| \ {\rm and}\ |x'_2|\cdot|x''_1|=|y'_2|\cdot|y''_1|.$$
  By the definition of the bilinear forms $\langle, \rangle$ and $[,]$, we have
  \begin{equation}\label{eq31}
    \langle x'_2, x''_1\rangle+[x'_2, x''_1]=\langle x''_1, x'_2\rangle+[x''_1, x'_2].
  \end{equation}
Thus, both $C$ and $D$ are equal to
\begin{eqnarray*}
 % & & \langle|x''_1|,|x'_2|\rangle-\langle |x'_2|,|x''_1|\rangle+2([|x'_1|+|x''_1|,|x'_2|+|x''_2|]+
%   [|x'_1|,|x''_1|]+[|x'_2|,|x''_2|])\\
   2[|x'_1|,|x'_2|]+2[|x'_1|,|x''_2|]+[|x''_1|,|x'_2|]+[|x'_2|,|x''_1|]+2[|x''_1|,|x''_2|]
   +2[|x'_1|,
   |x''_1|]+2[|x'_2|,|x''_2|].
\end{eqnarray*}
  Part (d) follows. Finally, the uniqueness of the bilinear form follows from (b), (c) and (d), and the symmetry of (,) follows from the uniqueness of
  (,).
\end{proof}
\subsection{The bialgebra $\mathfrak{f}$}\label{sec2.3}
 Let $\mathfrak{I}$ be the radical of the bilinear form (,). By an argument exactly the same as that in Section 1.2.3 in \cite{Lusztigbook}, we have
\begin{lem}\label{lem8}
$\mathfrak{I}$ is a two-sided ideal of ${}'\! \mathfrak{f}$.
\end{lem}
Let $\mathfrak{f}={}'\! \mathfrak{f}/\mathfrak{I}$ be the quotient algebra of ${}'\! \mathfrak{f}$ by the ideal $\mathfrak{I}$. By (\ref{eq30}), $\mathfrak{I}$ is $\mathbb{N}^I$-graded.
 This implies that $\mathfrak{f}$ is also an $\mathbb{N}^I$-graded algebra over $\mathbb{Q}(v,t)$. By abuse of notation, we denote again by $\theta_i$ the image of $\theta_i$ in $\mathfrak{f}$ under the quotient map.
 Moreover, the bilinear form (,) on ${}'\! \mathfrak{f}$ induces a well-defined symmetric bilinear form, denoted again by (,), since $\mathfrak{I}$ is the radical of (,).

We claim that the bilinear form on $\mathfrak{f}$ is non-degenerate. Assume that it is not, then there exists a nonzero element, say $x$, in $\mathfrak{f}$ such that $(x,y)=0$ for all $y \in \mathfrak{f}$. Let $x' \in {}'\! \mathfrak{f}$ be a representative of $x$, then $(x'+\mathfrak{I}, z+\mathfrak{I})=0$ for all $z \in {}'\! \mathfrak{f}$. So $x'\in \mathfrak{I}$. A contradiction. The claim follows.

 We claim that the radical of the bilinear form
on ${}'\! \mathfrak{f}\otimes {}'\! \mathfrak{f}$ in Proposition \ref{prop14} is $\mathfrak{I} \otimes {}'\! \mathfrak{f}+{}'\! \mathfrak{f}\otimes \mathfrak{I}$. Assume that $x\otimes y$ is in the radical of the bilinear form on ${}'\! \mathfrak{f}\otimes {}'\! \mathfrak{f}$, then for any element $x' \otimes y' \in {}'\! \mathfrak{f}\otimes {}'\! \mathfrak{f}$, we have
$$(x \otimes y, x'\otimes y')=t^{2[|x|,|y|]}(x,x')(y, y')=0.$$
Hence $(x,x')=0$ or $(y,y')=0$, i.e., $x\in \mathfrak{I}$ or $y\in \mathfrak{I}$. Thus $x\otimes y \in \mathfrak{I} \otimes {}'\! \mathfrak{f}+{}'\! \mathfrak{f}\otimes \mathfrak{I}$. The claim follows.

Moreover, we have
\begin{equation}\label{eq33}
  r(\mathfrak{I}) \subset \mathfrak{I} \otimes {}'\! \mathfrak{f}+{}'\! \mathfrak{f}\otimes \mathfrak{I}.
\end{equation}
Indeed, if $x \in \mathfrak{I}$, then we have
$(r(x), y\otimes z)=(x, yz)=0,\quad  \forall y,z \in {}'\! \mathfrak{f}.$ This implies that $r(x)$ is in the radical of (,) on ${}'\! \mathfrak{f}\otimes {}'\! \mathfrak{f}$.

By (\ref{eq33}), the map $r$ induces an algebra  homomorphism $ \mathfrak{f} \rightarrow  \mathfrak{f}\otimes  \mathfrak{f}$, denoted by $r$ again.
Here the algebra structure on $\mathfrak{f}\otimes  \mathfrak{f}$ is defined by equation (\ref{eq25}).
It is clear that $r(\theta_i)=\theta_i \otimes 1+1 \otimes \theta_i$ for any $\theta_i \in \mathfrak{f}$. So the coassociativity property in Lemma \ref{lem7} still holds for $\mathfrak{f}$.

 Let $^-: \mathbb{Q}(v,t) \rightarrow \mathbb{Q}(v,t)$ be the unique $\mathbb{Q}(t)$-algebra involution such that
\begin{equation}\label{eq73}
\overline{v}=v^{-1}\ {\rm and}\ \overline{t}=t.
\end{equation}
Let $^-: {}'\! \mathfrak{f}\rightarrow {}'\! \mathfrak{f}$ be the unique $\mathbb{Q}(t)$-algebra involution such that
$$\overline{p \theta_i}=\overline{p} \theta_i\ {\rm for\ all}\ p\in \mathbb{Q}(v,t)\ {\rm and}\ i \in I.$$
From the definition of $``^-"$ on ${}'\! \mathfrak{f}$, we have $|\overline{x}|=|x|$ for any homogeneous element
$x \in {}'\! \mathfrak{f}$.
\begin{lem}\label{lem9}
If $x\in {}'\! \mathfrak{f}$ is a homogeneous element and $r(x)=\sum x_1\otimes x_2$, then we have
$$r(\overline{x})=\sum v^{-|x_1|\cdot|x_2|}t^{\langle |x_2|, |x_1|\rangle-\langle |x_1|,|x_2|\rangle}\overline{x}_2\otimes \overline{x}_1.$$
\end{lem}
\begin{proof}
  It is clear that the lemma holds if $x=\theta_i$ for any $i \in I$. Assume that the lemma holds for the homogeneous elements $x'$ and $x''$. We shall show that the lemma holds for $x=x'x''$. Let us write
  $$r(x')=\sum x'_1\otimes x'_2\quad {\rm and}\quad  r(x'')=\sum x''_1\otimes x''_2,$$
  such that all factors are homogeneous.  By assumption, we have
  \begin{eqnarray*}
  r(\overline{x'})=\sum v^{-|x'_1|\cdot|x'_2|}t^{\langle |x'_2|, |x'_1|\rangle-\langle |x'_1|,|x'_2|\rangle}\overline{x'_2}\otimes \overline{x'_1},\ \
  %{\rm and}\
  r(\overline{x''})=\sum  v^{-|x''_1|\cdot|x''_2|}t^{\langle |x''_2|, |x''_1|\rangle-\langle |x''_1|,|x''_2|\rangle}\overline{x''_2}\otimes \overline{x''_1}.
  \end{eqnarray*}
  Hence, $r(\overline{x})=r(\overline{x'})r(\overline{x''})$ is equal to
  \begin{equation*}
\sum v^{-(|x'_1|+|x''_1|)\cdot (|x'_2|+|x''_2|)}t^{\langle|x'_2|+|x''_2|, |x'_1|+|x''_1|\rangle-\langle|x'_1|+|x''_1|, |x'_2|+|x''_2|\rangle}
   v^{|x''_1|\cdot |x'_2|}t^{\langle|x''_1|, |x'_2|\rangle-\langle|x'_2|,|x''_1|\rangle}\overline{x'_2x''_2}\otimes \overline{x'_1x''_1}.
  \end{equation*}
  On the other hand,
  $$r(x)=r(x'x'')=\sum v^{-|x'_2|\cdot|x''_1|}t^{\langle|x''_1|, |x'_2|\rangle-\langle|x'_2|,|x''_1|\rangle}x'_1x''_1\otimes x'_2x''_2.$$
  By (\ref{eq73}), we have
  $$\overline{v^{-|x'_2|\cdot|x''_1|}t^{\langle|x''_1|, |x'_2|\rangle-\langle|x'_2|,|x''_1|\rangle}x'_2x''_2}=v^{|x''_1|\cdot |x'_2|}t^{\langle|x''_1|,\ \ |x'_2|\rangle-\langle|x'_2|,|x''_1|\rangle}\overline{x'_2x''_2}.$$
   Since  $|x'_1x''_1|=|x'_1|+|x''_1|$ and $|x'_2x''_2|=|x'_2|+|x''_2|$,
   the lemma follows by induction on ${\rm tr}(|x|)$.
\end{proof}
\begin{lem}\label{lem10}
The involution map $^-:  {}'\! \mathfrak{f}\rightarrow  {}'\! \mathfrak{f}$ sends $\mathfrak{I}$ onto itself.
\end{lem}
\begin{proof}
Given any element $x\in \mathfrak{I}$, it is enough to show that $(\overline{x}, y)=0$ for any $y \in {}'\! \mathfrak{f}$. We can assume that $x$ is a homogenous element since $``^-"$ is additive. We shall show that
$(\overline{x}, y)=0$ by induction on ${\rm tr}(|x|)$. Without lost of generality, we assume that $y$ is a monomial and $y=y'y''$ for some monomials $y'$ and $y''$. Since $\theta_i \not \in \mathfrak{I}$, for any $i\in I$, we have ${\rm tr}(|x|)\geq 2$ for any element $x \in \mathfrak{I}$. Hence we can further assume that ${\rm tr}(|y'|), {\rm tr}(|y''|) >0$, i.e., $y'\not \in \mathbb{Q}(v,t)$ and $y''\not \in \mathbb{Q}(v,t)$.

Write $r(x)=\sum x_1\otimes x_2$. From (\ref{eq33}), $r(x) \in {}'\!\mathfrak{f}\otimes \mathfrak{I}+\mathfrak{I} \otimes {}'\!\mathfrak{f}$. Hence either $x_1 \in \mathfrak{I}$ or $x_2\in \mathfrak{I}$.
 If $x\in \mathfrak{I}$ satisfies that ${\rm tr}|x| \leq {\rm tr}|x'|$ for all $x'\in \mathfrak{I}$, then either $x_1 \in \mathbb{Q}(v,t)$ or $x_2 \in \mathbb{Q}(v,t)$, i.e., either ${\rm tr}(|x_1|)=0$ or ${\rm tr}(|x_2|)=0$. By Lemma \ref{lem9}, we have
\begin{eqnarray*}
(\overline{x}, y'y'')=(r(\overline{x}), y'\otimes y'')
=\sum v^{-|x_1|\cdot|x_2|}t^{\langle |x_2|,|x_1|\rangle-\langle |x_1|,|x_2|\rangle+2[|x_2|,|x_1|]}(\overline{x_2}, y')(\overline{x_1}, y'')=0.
\end{eqnarray*}
This shows that $r(\overline{x})\in \mathfrak{I}$ if  $x\in \mathfrak{I}$ and ${\rm tr}|x| \leq {\rm tr}|x'|$ for all $x'\in \mathfrak{I}$.

We now assume that $\overline{z}\in \mathfrak{I}$ for any $z\in \mathfrak{I}$ such that ${\rm tr}(|z|) <{\rm tr}(|x|)$. By Lemma \ref{lem9} again,
\begin{eqnarray*}
(\overline{x}, y'y'')=(r(\overline{x}), y'\otimes y'')
=\sum v^{-|x_1|\cdot|x_2|}t^{\langle |x_2|,|x_1|\rangle-\langle |x_1|,|x_2|\rangle+2[|x_2|,|x_1|]}(\overline{x_2}, y')(\overline{x_1}, y'').
\end{eqnarray*}
We may assume that both ${\rm tr}(|x_1|)<{\rm tr}(|x|)$ and ${\rm tr}(|x_2|)<{\rm tr}(|x|)$ by the assumption ${\rm tr}(|y'|)>0$ and ${\rm tr}(|y''|)>0$. Therefore, by the induction assumption, $\overline{x_1}, \overline{x_2} \in \mathfrak{I}$. This implies that $(\overline{x}, y'y'')=0$. This finishes the proof.
\end{proof}
By Lemma \ref{lem10}, the involution $^-: {}'\!\mathfrak{f} \rightarrow {}'\!\mathfrak{f}$ induces an involution on $\mathfrak{f}$, denoted by the same notation.

\subsection{Quantum Serre relations}

For any $i\in I$, let $r_i: {}'\!\mathfrak{f}\rightarrow {}'\!\mathfrak{f}$ be the unique linear map satisfying the following properties:
\begin{eqnarray*}
r_i(1)=0,\ \ r_i(\theta_j)=\delta_{ij} \quad \forall j\in I, \quad
{\rm and}\quad
r_i(xy)=v^{- i\cdot |y|}t^{\langle |y|,i\rangle-\langle i, |y|\rangle}r_i(x)y+xr_i(y),
\end{eqnarray*}
for any homogeneous elements $x$ and $y$. If we write $r(x)=\sum x_1\otimes x_2$ with $x_1,x_2$ homogenous and $x_2$'s of different degree, then
we have
\begin{eqnarray}\label{eq36}
x_1=r_i(x),\quad {\rm if}\ x_2=\theta_i.
  %r(x)&=&r_i(x)\otimes \theta_i\ {\rm plus\ homogeneities\ terms\ of}\\
%  &&{\rm different\ degree\ at\ the\ second\ component}\nonumber.
\end{eqnarray}
Since both $r$ and $r_i$ are linear maps, it is enough to check this
by assuming that $x$ is a monomial. This can be done by induction on $tr(|x|)$. If $|x|_i=0$, then $r_i(x)=0$ and there is no term of the form $-\otimes \theta_i$ in $r(x)$. The claim holds in this case. Now assume that $|x|_i\not =0$, then we can write $x=x'\theta_i x''$ for some
monomials $x',x''$ such that $|x''|_i=0$. So,
\begin{eqnarray*}
 & & r_i(x)=r_i(x'\theta_ix'')
 =v^{-i\cdot |x''|}t^{\langle |x''|,i\rangle-\langle i, |x''|\rangle}r_i(x'\theta_i)x''\\
 &=&v^{-i\cdot |x''|-i\cdot i}t^{\langle |x''|,i\rangle-\langle i, |x''|\rangle}r_i(x')\theta_ix''+v^{-i\cdot |x''|}t^{\langle |x''|,i\rangle-\langle i, |x''|\rangle}x'x''.
\end{eqnarray*}
On the other hand,
$$r(x)=r(x')(\theta_i\otimes 1+1\otimes \theta_i)r(x'').$$
Since $|x''|_i=0$,  the term $-\otimes \theta_i$ only appears in $(x'\otimes 1)(1\otimes \theta_i)(x''\otimes 1)+(z\otimes \theta_i)(\theta_i\otimes 1)(x''\otimes 1)$ for some $z$. By the induction assumption, $z=r_i(x')$. By (\ref{eq25}), the term
$-\otimes \theta_i$ is
$$(v^{-i\cdot |x''|-i\cdot i}t^{\langle |x''|,i\rangle-\langle i, |x''|\rangle}r_i(x')\theta_ix''+v^{-i\cdot |x''|}t^{\langle |x''|,i\rangle-\langle i, |x''|\rangle}x'x'')\otimes \theta_i.$$
The claim follows.

Similarly, for any $i\in I$, there is a unique linear map $_ir: {}'\!\mathfrak{f}\rightarrow {}'\!\mathfrak{f}$ satisfying the following properties:
\begin{eqnarray*}
_ir(1)=0,\ \ _ir(\theta_j)=\delta_{ij},\quad \forall j\in I,\quad
{\rm and}\quad
_ir(xy)={}_ir(x)y+v^{- i\cdot |x|}t^{\langle |x|,i\rangle-\langle i, |x|\rangle}x\ {}_ir(y),
\end{eqnarray*}
for any homogeneous elements $x, y$. Moreover, we have $r(x)=\theta_i\otimes {}_ir(x)$ plus terms of other bihomogeneities.
\begin{lem}\label{lem12}
For any $i\in I$, the linear maps $_ir, r_i: {}'\!\mathfrak{f} \rightarrow {}'\!\mathfrak{f}$ send $\mathfrak{I}$ to itself.
\end{lem}
\begin{proof}
  For any $x\in \mathfrak{I}$, if $|x|_i=0$, then $r_i(x)=0\in \mathfrak{I}$. The lemma holds in this case.
  We now assume that $|x|_i\not =0$. Write $r(x)=\sum x_1\otimes x_2$. By (\ref{eq36}), $r_i(x)\otimes \theta_i$ is one of the summands. By (\ref{eq33}), either $x_1\in \mathfrak{I}$ or $x_2\in \mathfrak{I}$. What follows is that $r_i(x)\in \mathfrak{I}$ since $\theta_i \not \in \mathfrak{I}$. It is similar to prove that $_i r(x)\in \mathfrak{I}$.
\end{proof}
\begin{lem}\label{lem13}
  For any $x, y\in {}'\!\mathfrak{f}$, we have
  \begin{equation}\label{eq34}
    (y\theta_i, x)=t^{2[|y|,i]}(y, r_i(x))(\theta_i, \theta_i),\quad  (\theta_i y, x)=t^{2[i,|y|]}(\theta_i, \theta_i)(y, {}_ir(x)).
  \end{equation}
\end{lem}
\begin{proof}
  By the properties of $r_i$, we have
  \begin{eqnarray*}
    (y\theta_i, x)=(y\otimes \theta_i, r(x))
    =(y\otimes \theta_i, r_i(x)\otimes \theta_i)
    =t^{2[|y|,i]}(y, r_i(x))(\theta_i, \theta_i),
  \end{eqnarray*}
  where the last equality is due to (\ref{eq74}). The first one follows. The second one can be proved similarly.
\end{proof}

By Lemma \ref{lem12}, $_ir$ and $r_i$ induce well-defined linear maps on $\mathfrak{f}$,
 denoted again by the same notations, respectively. Moreover, the property (\ref{eq34}) holds in $\mathfrak{f}$.
\begin{lem}\label{lem14}
  Let $x\in \mathfrak{f}_{\nu}$ with $\nu\not=0$. %with ${\rm tr}(|x|)\not =0$.
  \begin{itemize}
    \item[(a)] If $r_i(x)=0$ for all $i\in I$, then $x=0$.
    \item[(b)] If $_ir(x)=0$ for all $i\in I$, then $x=0$.
  \end{itemize}
\end{lem}
\begin{proof}
  If $r_i(x)=0$ for all $i\in I$, then, by Lemma \ref{lem13}, we have $(y\theta_i, x)=0$ for all $y$ and $\theta_i$.
   For any $z\in \mathfrak{f}_{\nu}$ with $\nu\not =0$, we have $z \in \sum_i \mathfrak{f}\theta_i$. Therefore $(z,x)=0$ for any $z\in \mathfrak{f}_{\nu}$. This implies $x$ is inside the radical of (,) on $\mathfrak{f}$. But (,) on $\mathfrak{f}$ is non-degenerate, so $x=0$. This finishes the proof of (a). (b) can be proved similarly.
\end{proof}

For any $n\in \mathbb{N}$, we set %the divided power $\theta_i^{(n)}$of $\theta_i$ is defined to be
 $$\theta_i^{(n)}=\frac{\theta_i^n}{[n]^!_{v_i, t_i}},$$
where $v_i$ and $t_i$ are defined in Section \ref{sec2.01}.
\begin{lem}\label{lem16}
We have
$r(\theta_i^{(n)})=\sum_{p+p'=n}(v_it_i)^{-pp'}\theta^{(p)}_i \otimes \theta^{(p')}_i$, for any $n\in \mathbb{N}$.
\end{lem}
\begin{proof} Since $(1\otimes \theta_i)(\theta_i\otimes 1)=v_i^{-2}(\theta_i\otimes 1)(1\otimes \theta_i)$ and by Section 1.3.5 in \cite{Lusztigbook}, we have
$$r(\theta_i^n)=(1\otimes \theta_i+\theta_i\otimes 1)^n=\sum_{p+p'=n} (v_it_i)^{-pp'}\begin{bmatrix}
    n\\p
  \end{bmatrix}_{v_i, t_i}\theta_i^p\otimes \theta_i^{p'}.$$
  The lemma follows from the definitions of $\theta_i^{(n)}$ and $\begin{bmatrix}
    n\\p
  \end{bmatrix}_{v_i, t_i}$.
\end{proof}

For any $n\in \mathbb{N}$ and $i\in I$, by Lemma \ref{lem16} and (\ref{eq36}), we have
\begin{equation}\label{eq93}
r_i(\theta_i^{(n+1)})=(v_it_i)^{-n}\theta^{(n)}_i.
\end{equation}

\begin{prop}\label{prop15}
  The generators $\theta_i$ of $\mathfrak{f}$ satisfy the following identity.
  $$S_{ij}:=\sum_{p+p'=1-2\frac{i\cdot j}{i\cdot i}}(-1)^pt_i^{-p(p'-2\frac{\langle i,j\rangle}{i\cdot i}+2\frac{\langle j,i\rangle}{i\cdot i})}\theta_i^{(p)}\theta_j\theta_i^{(p')}=0, \quad \forall i\not =j\in I.$$
\end{prop}
\begin{proof}
  We set $a'=-2\frac{\langle i,j\rangle}{i\cdot i},\ a''=-2\frac{\langle j,i\rangle}{i\cdot i}$ and $N=-2\frac{i\cdot j}{i\cdot i}$. So $N=a'+a''$ and we can rewrite $S_{ij}$ as follows:
  $$S_{ij}=\sum_{p+p'=N+1}(-1)^pt_i^{-p(p'+a'-a'')}
  \theta_i^{(p)}\theta_j\theta_i^{(p')}.$$
  By Lemma \ref{lem14}, we only need to show that $r_k(S_{ij})=0$ for any $k\in I$. It is clear that
  \begin{equation}\label{eq97}
  r_k(S_{ij})=0\quad {\rm if\ } k\not =i, j.
  \end{equation}
  By (\ref{eq93}) and the definition of $r_i$, we have
  \begin{eqnarray*}
    & &r_i(\theta_i^{(p)}\theta_j\theta_i^{(p')})
   =v^{-2p'\langle i,i\rangle}r_i(\theta_i^{(p)}\theta_j)\theta_i^{(p')}+(v_it_i)^{1-p'}
    \theta_i^{(p)}\theta_j\theta_i^{(p'-1)}\\
   &=&v_i^{-2p'}v^{-i\cdot j}t^{\langle j,i\rangle-\langle i,j\rangle}(v_it_i)^{1-p}\theta_i^{(p-1)}\theta_j\theta_i^{(p')}+(v_it_i)^{1-p'}
    \theta_i^{(p)}\theta_j\theta_i^{(p'-1)}\\
   &=&v_i^{-2p'+a'+a''+1-p}t_i^{a'+1-p-a''}\theta_i^{(p-1)}\theta_j\theta_i^{(p')}+
   (v_it_i)^{1-p'}\theta_i^{(p)}\theta_j\theta_i^{(p'-1)}.
  \end{eqnarray*}
So  $r_i(S_{ij})$ is equal to
\begin{eqnarray}\label{eq75}
   \sum_{1\leq p\leq N+1}(-1)^pt_i^{A_1}
   v_i^{A_2}t_i^{A_3}\theta_i^{(p-1)}\theta_j\theta_i^{(p')}
   +\sum_{0\leq p\leq N}(-1)^pt_i^{A_1} (v_it_i)^{1-p'}\theta_i^{(p)}\theta_j\theta_i^{(p'-1)}\\
   = \sum_{0\leq p\leq N}(-1)^{p+1}t_i^{A_4}
   v_i^{A_2+1}t_i^{A_3-1}\theta_i^{(p)}\theta_j\theta_i^{(p'-1)}
   +\sum_{0\leq p\leq N}(-1)^pt_i^{A_1} (v_it_i)^{1-p'}\theta_i^{(p)}\theta_j\theta_i^{(p'-1)},\nonumber
\end{eqnarray}
where $A_1=-p(p'+a'-a'')$, $A_2=-2p'+a'+a''+1-p$, $A_3=a'+1-p-a''$ and $A_4=-(p+1)(p'-1+a'-a'')$.
Since $a'+a''=N$ and $p+p'=N+1$, by comparing the exponents of $v_i$ and $t_i$ in (\ref{eq75}), we have \begin{equation}\label{eq98}
  r_i(S_{ij})=0.
\end{equation}

  By  (\ref{eq93}) and the definition of $r_j$ again, we have
  \begin{equation*}
r_j(\theta_i^{(p)}\theta_j\theta_i^{(p')})
%=v^{-p'j\cdot i}t^{p'(\langle i,j\rangle-\langle j,i\rangle)} r_j(\theta_i^{(p)}\theta_j)\theta_i^{(p')}
     = v^{-p'j\cdot i}t^{p'(\langle i,j\rangle-\langle j,i\rangle)}  \theta_i^{(p)}\theta_i^{(p')}
     =v_i^{p'(a'+a'')}t_i^{p'(a''-a')}\theta_i^{(p)}\theta_i^{(p')}.
  \end{equation*}
  So $r_j(S_{ij})$ is equal to
  \begin{eqnarray*}
& &\sum_{p+p'=N+1}(-1)^pt_i^{-p(p'+a'-a'')}v_i^{p'(a'+a'')}t_i^{p'(a''-a')}
   \theta_i^{(p)}\theta_i^{(p')}\\
   &=&\sum_{p+p'=N+1}(-1)^pt_i^{-p(p'+a'-a'')}v_i^{p'(a'+a'')}t_i^{p'(a''-a')}\begin{bmatrix}
     N+1\\p
   \end{bmatrix}_{v, t}\theta_i^{(N+1)}.
  \end{eqnarray*}
  By (\ref{eq92}), to show that $r_j(S_{ij})=0$, it is enough to show that
  \begin{equation}\label{eq64}
\sum_{p+p'=N+1}(-1)^pt_i^{-p(p'+a'-a'')}v_i^{p'(a'+a'')}t_i^{p'(a''-a')+pp'}\begin{bmatrix}
     N+1\\p
   \end{bmatrix}_{v}=0.
  \end{equation}
   By using $a'+a''=N$ and $p+p'=N+1$, the left hand side of (\ref{eq64}) is
   $$t_i^{(N+1)(a''-a')}\sum_{p+p'=N+1}(-1)^pv_i^{p(a'+a'')}\begin{bmatrix}
     N+1\\p
   \end{bmatrix}_{v}.$$
   By Section 1.3.4 in \cite{Lusztigbook}, this is 0. So we have
   \begin{equation}\label{eq99}
     r_j(S_{ij})=0.
   \end{equation}
   By (\ref{eq97}), (\ref{eq98}), (\ref{eq99}), we have $r_k(S_{ij})=0,\ \forall k \in I$. This finishes the proof.
\end{proof}
In Section \ref{sec2.6}, we shall show that the ideal $\mathfrak{J}$ is generated by $\{S_{ij}, i \neq j\}$.

\subsection{Comparison}
\label{Comparison}
In this section, we compare the algebra $\mathfrak{f}$ with various versions of quantum algebras in literature.

(a). Two-parameter quantum algebras are defined case by case in \cite{benkart2001rep} and \cite{hu2012notes}. If we set
%$v=(\alpha\beta)^{\frac{1}{2}}$ and $t=(\alpha\beta^{-1})^{\frac{1}{2}}$
 $v=(rs^{-1})^{\frac{1}{2}}$ and $t=(rs)^{-\frac{1}{2}}$, then the quantum Serre relation for $\mathfrak{f}$ coincides with the one
in \cite{hu2012notes} and  \cite{benkart2001rep}.

(b). Given a Dynkin quiver, Reineke defines a $\mathbb{Q}(\alpha,\beta)$-algebra $\mathcal{H}_{\alpha, \beta}$ in \cite{Reineke2001}.
 By Proposition 6.3 in \cite{Reineke2001}, $\mathcal{H}_{\alpha, \beta}$ is isomorphic to the positive part of the two-parameter quantum algebra in \cite{hu2012notes} associated to the Dynkin quiver.
 Let $\Omega=Id-A$, $v=(\alpha\beta)^{\frac{1}{2}}$ and $t=(\alpha\beta^{-1})^{\frac{1}{2}}$, where $Id$ is the identity matrix and $A$ is the adjacent matrix of  the Dynkin quiver. Then $\mathcal{H}_{\alpha,\beta}$ is isomorphic to $\mathfrak{f}$.

 Theorem \ref{thm8} in Section \ref{sec5.6} shows that $\mathfrak{K}$ is a geometrization of ${}_{\mathfrak{A}}\mathfrak{f}$. By the sheaf-function correspondence, we obtain a Hall-algebra construction of $\mathfrak f$ associated to arbitrary $\Omega$, not just the one associated to a Dynkin quiver. In a sense, this generalizes Reineke's construction.

(c). In \cite{Humultipara}, Hu, Pei and Rosso define quantum algebras with multi-parameters $q_{ij}$ associated to $(I,\cdot)$.
 Let us recall the multi-parameter quantum Serre relations from \cite{Humultipara}. It is
\begin{equation}\label{eq39}
\sum_{p+p'=1-i \cdot j}(-1)^pq_{ii}^{-\frac{pi \cdot j}{2}}q_{ji}^p \begin{bmatrix}
  1-i \cdot j\\p
\end{bmatrix}_{q_{ii}^{\frac{1}{2}}}\theta_i^{p}\theta_j\theta_i^{p'}=0.
\end{equation}
Since we use different Gaussian binomial coefficients, (\ref{eq39}) is slightly different from the original one in \cite{Humultipara}.
On the other hand, we can rewrite the quantum Serre relations in Proposition \ref{prop15} as follows.
\begin{equation}\label{eq40}
\sum_{p+p'=1-i \cdot j}(-1)^pt^{p(\langle i,j\rangle-\langle j,i\rangle)}\begin{bmatrix}
  1-i \cdot j\\p
\end{bmatrix}_{v_i}
  \theta_i^{p}\theta_j\theta_i^{p'}=0.
\end{equation}
By setting $q_{ii}=v^{2}_i$ and $q_{ji}=v^{i\cdot j}t^{\langle i,j\rangle-\langle j,i\rangle}$, (\ref{eq39}) is reduced to (\ref{eq40}). In other words, $\mathfrak{f}$ is a specialization of a multi-parameter quantum algebra defined in \cite{Humultipara}.

\section{Specialization and deformation}
\subsection{Negative part}\label{sec2.6}

Recall that $(I, \cdot)$ is the Cartan datum associated to the matrix $\Omega$.
If we set $t=1$, the construction in Section \ref{sec2} is exactly Lusztig's construction in \cite{Lusztigbook}. In particular, the specialization of the bialgebra $(\mathfrak{f}, \cdot, r)$ in Section \ref{sec2.3} at $t=1$ is Lusztig's algebra in \cite{Lusztigbook} associated to $(I,\cdot)$, which we shall denote  by $(\mathbf{f}, \circ, \widetilde{r}_1)$,
where $\circ$ and  $\widetilde{r}_1$ are multiplication and comultiplication of $\mathbf{f}$, respectively.

Besides specialization, there is another way to relate these two bialgebras. Namely, we shall show that $(\mathbf{f}, \circ, \widetilde{r}_1)$  can be deformed to $\mathfrak{f}$. Let $\mathbf{f}_{v,t}=\mathbf{f} \otimes_{\mathbb{Q}(v)}\mathbb{Q}(v,t)$. The bialgebra structure on $\mathbf{f}$ can be naturally extended to $\mathbf{f}_{v,t}$, denoted again by $(\circ, \widetilde{r}_1)$.

We define a new multiplication $``\odot"$ on $\mathbf{f}_{v,t}$ by
\begin{equation}\label{eq80}
  x \odot y=t^{[ |x|,\ |y|]}x \circ y,
\end{equation}
for any homogenous elements $x, y \in \mathbf{f}_{v,t}$, where [,] is defined in (\ref{eq48}).

Define a new multiplication, denote again by $\odot$, on $\mathbf{f}_{v,t}\otimes \mathbf{f}_{v,t}$ as follows.
\begin{equation}\label{eq82}
(x_1 \otimes x_2)\odot (y_1 \otimes y_2)=v^{-|y_1|\cdot|x_2|}t^{\langle|y_1|,|x_2|\rangle-\langle |x_2|,|y_1|\rangle}(x_1\odot y_1) \otimes (x_2\odot y_2),
\end{equation}
for any homogeneous elements $x_1,x_2,y_1$ and $y_2$.

 For any $x \in \mathbf{f}_{v,t}$, we write $\widetilde{r}_1(x)=\sum x_1\otimes x_2$ with $x_1$, $x_2$ homogenous. Define a linear map $r_1: \mathbf{f}_{v,t}\rightarrow \mathbf{f}_{v,t}\otimes \mathbf{f}_{v,t}$ by
 \begin{equation}\label{eq81}
r_1(x)=\sum t^{-[|x_1|, |x_2|]} x_1\otimes x_2.
 \end{equation}

\begin{prop}\label{prop25}
The linear map $r_1: (\mathbf{f}_{v,t},\odot) \rightarrow (\mathbf{f}_{v,t}\otimes \mathbf{f}_{v,t}, \odot)$ is an algebra homomorphism. % with respect to $\odot$ on both sides.
 \end{prop}
\begin{proof}
For any homogenous elements $x$ and $y$ in $\mathbf{f}_{v,t}$, we write $\widetilde{r}_1(x)=\sum x_1 \otimes x_2$ and $\widetilde{r}_1(y)=\sum y_1 \otimes y_2$ with $x_1,x_2, y_1$ and $y_2$ are all homogenous. Then $\widetilde{r}_1(x\circ y)=\sum v^{-|x_2| \cdot |y_1|}(x_1\circ y_1) \otimes (x_2\circ y_2)$. Therefore,
\begin{equation}\label{eq83}
  r_1(x \odot y)=t^{[|x|,|y|]} r_1(x \circ y)=\sum v^{-|x_2| \cdot |y_1|}t^{[|x|,|y|]}
  t^{-[|x_1|+|y_1|, |x_2|+|y_2|]}(x_1\circ y_1) \otimes (x_2\circ y_2).\nonumber
\end{equation}
On the other hand, we have
\begin{eqnarray}\label{eq84}
 & & r_1(x) \odot r_1(y)=\sum t^{-[|x_1|,|x_2|]-[|y_1|,|y_2|]}(x_1\otimes x_2)\odot (y_1\otimes y_2)\\
  &=&\sum t^{-[|x_1|,|x_2|]-[|y_1|,|y_2|]}v^{-|x_2| \cdot |y_1|}t^{\langle|y_1|,|x_2|\rangle-\langle|x_2|,|y_1|\rangle}(x_1\otimes x_2)\odot (y_1\otimes y_2).\nonumber
\end{eqnarray}
By comparing (\ref{eq83}) with (\ref{eq84}), it is reduced to check that
$$[|x|,|y|]-[|x_1|+|y_1|, |x_2|+|y_2|]=\langle|y_1|,|x_2|\rangle-\langle|x_2|,|y_1|\rangle-[|x_1|,|x_2|]-[|y_1|,|y_2|],$$
which follows from (\ref{eq31}).
\end{proof}
\begin{thm}\label{thm16}
 The assignment $\theta_i \mapsto \theta_i, \forall i\in I$ gives a twisted bialgebra isomorphism $\mathfrak{f}\simeq (\mathbf{f}_{v,t}, \odot, r_1)$.
\end{thm}
\begin{proof} Recall that $ {}'\!\mathfrak{f}$ is the free algebra generated by $\theta_i, i\in I$. Let ${}'\!\phi: {}'\!\mathfrak{f} \rightarrow \mathbf{f}_{v,t}$ be the algebra homomorphism sending $\theta_i$ to $\theta_i$, where the algebra structure of $\mathbf{f}_{v,t}$ is defined in (\ref{eq80}). Consider the following diagram.
 \begin{equation}\label{eq86}
\xymatrix{{}'\!\mathfrak{f} \ar[r]^-{{}'\!\phi} \ar[d]^{r} & \mathbf{f}_{v,t}\ar[d]^{r_1}\\
 {}'\!\mathfrak{f}\otimes {}'\!\mathfrak{f} \ar[r]^-{{}'\!\phi\otimes {}'\!\phi} &\mathbf{f}_{v,t}\otimes \mathbf{f}_{v,t}.
   }
 \end{equation}
   Since all maps are algebra homomorphisms, the diagram commutes by checking the image on $\theta_i, \ \forall i\in I$.

Recall that there is a unique non-degenerate bilinear form, $(,)_L$, on $\mathbf{f}$ defined in Chapter 1 in \cite{Lusztigbook} satisfying the following properties.
\begin{itemize}
    \item[(a)] $(1,1)_L=1,\quad (\theta_i, \theta_j)_L=\delta_{ij}\frac{1}{1-v^{-2}_i}$\quad for all $i, j \in I$,
    \item[(b)] $(x, y'\circ y'')_L=(\widetilde{r}_1(x), y' \otimes y'')_L$\quad for all $x, y', y'' \in \mathbf{f}$,
    \item[(c)] $(x'\circ x'', y)_L=(x'\otimes x'', \widetilde{r}_1(y))_L$\quad for all $x', x'', y \in \mathbf{f}$.
  \end{itemize}
  Since any element $x\in \mathbf{f}_{v,t}$ can be written into $a=\sum a_i \otimes t^i$, we can extend the bilinear form to $\mathbf{f}_{v,t}$
   by setting $(x \otimes t^m, y\otimes t^n)_L=t^{m+n}(x,y)_L$. Moreover, this bilinear form on $\mathbf{f}_{v,t}$ still satisfies the above properties (a), (b), (c).

Now we define a new bilinear form, $(,)'$, on ${}'\!\mathfrak{f}$ as follows.
\[(x,y)'=({}'\!\phi(x), {}'\!\phi(y))_L,\ \forall x, y \in {}'\!\mathfrak{f}.\]
We claim that  ${\rm Rad} (,)'={\rm Ker}({}'\!\phi)$, where ${\rm Rad} (,)' $ is the radical of $ (,)'$.
 It is clear that ${\rm Ker}({}'\!\phi)\subset {\rm Rad} (,)' $. Now for any $z\in \mathbf{f}_{v,t}$, there is $y\in {}'\!\mathfrak{f}$ such that ${}'\!\phi(y)=z$ since ${}'\!\phi$ is a surjective map. Therefore, for any $x\in {\rm Rad} (,)' $, $0=(x,y)'=({}'\!\phi(x), {}'\!\phi(y))_L=({}'\!\phi(x),z)_L$. Since $(,)_L$ is non-degenerate and $z$ is arbitrary, ${}'\!\phi(x)=0$. So $x\in {\rm Ker}({}'\!\phi).$

To prove that $\mathfrak{f}\simeq (\mathbf{f}_{v,t}, \odot)$, it is enough to show that $(,)'$ satisfying the properties (a),(b),(c),(d) in Proposition \ref{prop14}. The bilinear form on ${}'\!\mathfrak{f} \otimes {}'\!\mathfrak{f}$ is defined by
  \begin{equation}\label{eq85}
    (x_1 \otimes x_2, y_1 \otimes y_2)'=t^{2[|x_1|,|x_2|]}(x_1, y_1)' (x_2, y_2)'.
  \end{equation}
  Properties (a) and (b) are obvious. We now check property (d). For $y \in {}'\!\mathfrak{f}$, we write $r(y)=\sum y'\otimes y''$ and $\widetilde{r}_1({}'\!\phi(y))=\sum y_1\otimes y_2$. Then for any $a, b \in {}'\!\mathfrak{f}$, we have
  \begin{eqnarray*}
    &(a b,y)'=({}'\!\phi(a b), {}'\!\phi(y))_L=({}'\!\phi(a)\odot {}'\!\phi(b), {}'\!\phi(y))_L
    =t^{[|a|,|b|]}({}'\!\phi(a)\circ {}'\!\phi(b), {}'\!\phi(y))_L\\
   & =t^{[|a|,|b|]}({}'\!\phi(a)\otimes {}'\!\phi(b), \widetilde{r}_1({}'\!\phi(y)))_L
    =t^{[|a|,|b|]}\sum({}'\!\phi(a)\otimes {}'\!\phi(b), y_1\otimes y_2)_L.
  \end{eqnarray*}
  On the other hand, we have
  \begin{equation*}
  \begin{split}
    (a& \otimes b,r(y))'=\sum (a\otimes b, y'\otimes y'')'=\sum t^{2[|a|,|b|]}(a, y')'( b, y'')'\\
    &=\sum t^{2[|a|,|b|]}({}'\!\phi(a), {}'\!\phi(y'))_L( {}'\!\phi(b),{}'\!\phi( y''))_L
    =\sum t^{2[|a|,|b|]}({}'\!\phi(a)\otimes {}'\!\phi(b),{}'\!\phi(y')\otimes {}'\!\phi( y''))_L\\
   % &= t^{2[|a|,|b|]}({}'\!\phi(a)\otimes {}'\!\phi(b),\sum ({}'\!\phi\otimes {}'\!\phi)(y'\otimes y''))_L
   & =t^{2[|a|,|b|]}({}'\!\phi(a)\otimes {}'\!\phi(b),({}'\!\phi\otimes {}'\!\phi)(r(y)))_L
    =t^{2[|a|,|b|]}({}'\!\phi(a)\otimes {}'\!\phi(b),r_1({}'\!\phi(y)))_L\\
   & =t^{2[|a|,|b|]}({}'\!\phi(a)\otimes {}'\!\phi(b),\sum t^{-[|y_1|,|y_2|]} y_1\otimes y_2)_L
    =t^{[|a|,|b|]}\sum({}'\!\phi(a)\otimes {}'\!\phi(b), y_1\otimes y_2)_L.
  \end{split}
  \end{equation*}
  The last equality follows from the fact that $|y_1|=|a|$ and $|y_2|=|b|$. This proves that the property (d) and (c) can be proved similarly. The claim that ${\rm Rad} (,)'={\rm Ker}({}'\!\phi)$ holds.

   The coalgebra homomorphism follows from the commutativity of the Diagram (\ref{eq86}). This finishes the proof.
\end{proof}

On $\mathbf{f}_{v,t}$, we have two different bialgebra structures, $(\mathbf{f}_{v,t}, \circ, \widetilde{r}_1)$ and $(\mathbf{f}_{v,t}, \odot, r_1)$. By Theorem \ref{thm16}, we have $\mathfrak{f}\simeq (\mathbf{f}_{v,t}, \odot, r_1)$. We now define a second  bialgebra structure on $\mathfrak{f}$ corresponding to $(\mathbf{f}_{v,t}, \circ, \widetilde{r}_1)$. Define a new multiplication $``*"$ on $\mathfrak{f}$ by
\begin{equation}\label{eq89}
  x *y=t^{-[|x|, |y|]}xy,
\end{equation}
for any homogeneous elements $x$ and $y$. $(\mathfrak{f}, *)$ is an associative algebra due to the fact that $[,]$ is a bilinear form.
 Define a new multiplication, denote again by $*$, on $\!\mathfrak{f}\otimes \!\mathfrak{f}$ as follows.
\begin{equation}\label{eq55}(x_1 \otimes x_2)*(y_1 \otimes y_2)=v^{-|y_1|\cdot|x_2|}x_1*y_1 \otimes x_2*y_2,
\end{equation}
for any homogeneous elements $x_1,x_2,y_1$ and $y_2$.
 For any $x \in \mathfrak{f}$, we write $r(x)=\sum x_1\otimes x_2$ with $x_1$, $x_2$ homogenous. Define a linear map $\widetilde{r}: \mathfrak{f}\rightarrow \mathfrak{f}\otimes \mathfrak{f}$ by
 \begin{equation}\label{eq68}
\widetilde{r}(x)=\sum t^{[|x_1|, |x_2|]} x_1\otimes x_2.
 \end{equation}
By a similar argument as that of Proposition \ref{prop25}, we have that
 $\widetilde{r}: \mathfrak{f}\rightarrow \mathfrak{f}\otimes \mathfrak{f}$ is an algebra homomorphism with respect to the multiplications $*$ on both sides. By Theorem \ref{thm16}, we have

 \begin{prop}\label{prop28}
The assignment $\theta_i\mapsto \theta_i, \forall i \in I$ gives a twisted  bialgebra isomorphism
   $(\mathfrak{f}, *, \widetilde{r})\simeq (\mathbf{f}_{v,t},\circ, \widetilde{r}_1)$.
\end{prop}

Recall that $\mathfrak{f}={}'\mathfrak{f}/\mathfrak{J}$ and $\mathfrak{J}$ is the radical of the bilinear form of (,) on ${}'\mathfrak{f}$. By Proposition \ref{prop28}, we have

\begin{cor}\label{cor7}
  $\mathfrak{J}$ is generated by $S_{ij}, \forall i\not =j \in I$, where $S_{ij}$ is defined in Proposition \ref{prop15}.
\end{cor}

Suppose $\Omega'$ is another matrix satisfying (a),(b),(c) in Section \ref{sec2.1}. Let $\mathfrak{f}(\Omega')$ be the bialgebra constructed in Section \ref{sec2.3} associated to $\Omega'$. By Proposition \ref{prop28}, we have

\begin{cor}\label{cor6}
If the associated Cartan datums of $\Omega$ and $\Omega'$ are the same, then the assignment $\theta_i \mapsto \theta_i, \forall i\in I$ gives a twisted bialgebra isomorphism
 $(\mathfrak{f},*, \widetilde{r})\simeq (\mathfrak{f}(\Omega'),*, \widetilde{r})$.
\end{cor}

 Let $\phi: \mathfrak{f}\rightarrow (\mathbf{f}_{v,t},\odot)$ be the induced map from ${}'\!\phi: {}'\!\mathfrak{f}\rightarrow (\mathbf{f}_{v,t},\odot)$. This is an algebra isomorphism by Theorem \ref{thm16}. Moreover, one can easily check that
 \begin{equation}\label{eq95}
   \phi(a*b)=\phi(a)\circ \phi(b), \quad \forall a, b\in \mathfrak{f}.
 \end{equation}
  We define a new bilinear form, denoted by $(,)^*$, on $\mathfrak{f}$ and $\mathfrak{f}\otimes \mathfrak{f}$ by
 \begin{equation}\label{eq88}
   (a\otimes x, b\otimes y)^*=(\phi(a)\otimes \phi(x), \phi(b)\otimes \phi(y))_L,\ {\rm and}\
  (a, b)^*=(\phi(a), \phi(b))_L, \  \forall a,b,x, y\in \mathfrak{f}.%\nonumber
 \end{equation}
We notice that $(,)$ and $(,)^*$ are different on $\mathfrak{f}\otimes \mathfrak{f}$ by comparing (\ref{eq85}) with (\ref{eq88}).
Moreover, if we consider $\mathbf{f} \otimes 1$ as a subalgebra of $(\mathfrak{f},*)$ via the map $\phi$, then  the restriction of $(,)^*$ to $\mathbf{f} \otimes 1$ coincides with the bilinear form $(,)_L$ on $\mathbf{f}$ in the proof of Theorem \ref{thm16}. By (\ref{eq95}), (\ref{eq88}) and the property of $(,)_L$, we have
\begin{equation}\label{eq96}
(a*b, z)^*=(a\otimes b, \widetilde{r}(z))^*,\quad \forall a, b, z\in \mathfrak{f}.
\end{equation}

 \begin{prop}\label{prop23}
 For any $x,y \in \mathbf{f} \otimes 1$, we have $(x,y)=(x,y)^*.$
 \end{prop}
 \begin{proof}
   We show it by induction on $tr(|x|)$. Since both $(,)$ and $(,)^*$ are bilinear on $\mathbf{f} \otimes 1$, we can assume that both $x$ and $y$ are monomials. If $tr(|x|)=1$, then $x=\theta_i$ for some $i \in I$ and $(\theta_i, \theta_i)=(\theta_i, \theta_i)^*$.

  If $x \in \mathbf{f} \otimes 1$, by Theorem \ref{thm16}, $x *\theta_i=t^{-[|x|,i]}x\theta_i \in \mathbf{f}\otimes 1$.
  We now assume that $(x,y)=(x,y)^*$ for any $y\in \mathbf{f}\otimes 1$ with $|y|=|x|$.
  We want to show that $(x*\theta_i,z)=(x*\theta_i,z)^*$ for any $i\in I$ and any $z\in \mathbf{f}\otimes 1$ with $|z|=|x|+i$.
  By Lemma \ref{lem13}, we have
  \begin{equation}\label{eq69}
    (x*\theta_i,z)=t^{-[|x|,i]}(x\theta_i, z)=t^{[|x|,i]}(x, r_i(z))(\theta_i,\theta_i).
  \end{equation}
  On the other hand, by (\ref{eq36}) and (\ref{eq68}), we have $\widetilde{r}(x)=t^{[|x|,i]}r_i(x)\otimes \theta_i$ modulo homogeneities terms of different degree at the second component. Therefore, by (\ref{eq96}), we have
  \begin{equation}\label{eq70}
    (x*\theta_i,z)^*=(x\otimes \theta_i,\widetilde{r}(z))^*=t^{[|x|,i]}(x, r_i(z))^*(\theta_i,\theta_i)^*.
  \end{equation}
  By the induction assumption, (\ref{eq69}) and (\ref{eq70}) are equal. Proposition follows.
 \end{proof}

\subsection{Entire algebra}\label{sec6}
%\subsection{The algebra $U_{v,t}$}
\label{presentation}

 Recall that $\Omega=(\Omega_{ij})_{i,j\in I}$ is the matrix fixed in Section \ref{sec2.1}.
{\it The two-parameter quantum algebra} $U_{v,t}$ associated to $\Omega$ is an associative $\mathbb{Q}(v,t)$-algebra with 1 generated by symbols $E_i, F_i, K_i^{\pm 1}, K_i'^{\pm 1},$ $\forall i\in I$ and subject to the following relations.
\begin{eqnarray*}
  (R1)& & K_i^{\pm 1}K^{\pm 1}_j=K^{\pm 1}_jK_i^{\pm 1},\ \ K'^{\pm 1}_iK'^{\pm 1}_j=K'^{\pm 1}_jK'^{\pm 1}_i,\\
     & & K_i^{\pm 1}K'^{\pm 1}_j=K'^{\pm 1}_jK_i^{\pm 1},\ \ K_i^{\pm 1}K_i^{\mp 1}=1=K'^{\pm 1}_iK'^{\mp 1}_i.\\
  (R2)& &K_iE_jK^{-1}_i=v^{i\cdot j}t^{\langle i,j\rangle-\langle j,i\rangle}E_j,\ \ K'_iE_jK'^{-1}_i=v^{-i\cdot j}t^{\langle i,j\rangle-\langle j,i \rangle}E_j,\\
     & &K_iF_jK^{-1}_i=v^{-i\cdot j}t^{\langle j,i\rangle-\langle i,j\rangle}F_j,\ \ K'_iF_jK'^{-1}_i=v^{i\cdot j}t^{\langle j,i\rangle-\langle i,j\rangle}F_j.\\
  (R3)& &E_iF_j-F_j E_i=\delta_{ij}\frac{K_i-K'_i}{v_i-v^{-1}_i}.\\
  (R4)& & \sum_{p+p'=1-2\frac{i\cdot j}{i\cdot i}}(-1)^pt_i^{-p(p'-2\frac{\langle i,j\rangle}{i\cdot i}+2\frac{\langle j,i\rangle}{i\cdot i})}E_i^{(p')}E_j E_i^{(p)}=0,\quad {\rm if}\ i\not =j,\\
  & &\sum_{p+p'=1-2\frac{i\cdot j}{i\cdot i}}(-1)^pt_i^{-p(p'-2\frac{\langle i,j\rangle}{i\cdot i}+2\frac{\langle j,i\rangle}{i\cdot i})}F_i^{(p)}F_j F_i^{(p')}=0,\quad {\rm if}\ i\not =j,
\end{eqnarray*}
where $E^{(p)}_i=\frac{E_i^p}{[p]^!_{v_i,t_i}}$.
 The algebra $U_{v,t}$ has a Hopf algebra structure with the comultiplication $\Delta$, the counit $\varepsilon$ and the antipode $S$ given as follows.

  $$\begin{array}{llll}
   &\Delta(K_i^{\pm 1})=K_i^{\pm 1} \otimes K_i^{\pm 1},&\Delta(K'^{\pm 1}_i)=K'^{\pm 1}_i \otimes K'^{\pm 1}_i, &\vspace{4pt}\\
   &\Delta(E_i)=E_i\otimes 1+K_i\otimes E_i,& \Delta(F_i)=1 \otimes F_i+F_i\otimes K'_i, & \vspace{4pt}\\
  &\varepsilon(K_i^{\pm 1})=\varepsilon(K'^{\pm 1}_i)=1,& \varepsilon(E_i)=\varepsilon(F_i)=0,& S(K_i^{\pm 1})=K_i^{\mp 1},\vspace{4pt}\\
  & S(K'^{\pm 1}_i)=K'^{\mp 1}_i,&
  S(E_i)=-K_i^{-1}E_i,&S(F_i)=-F_iK'^{-1}_i.
  \end{array}$$
This can be proved by checking the above relations (R1)--(R4). We refer to Chapter 3 in \cite{Lusztigbook} for more details.

For any $\gamma=(\gamma_1, \gamma_2), \eta=(\eta_1,\eta_2) \in \mathbb{Z}^I \times \mathbb{Z}^I$, we define a bilinear form on $\mathbb{Z}^I \times \mathbb{Z}^I$ by
\begin{eqnarray*}
[ \gamma, \eta ]'=[\gamma_2, \eta_2]-[\gamma_1, \eta_1].
\end{eqnarray*}
 The algebra $U_{v,t}$  admits a $\mathbb{Z}^I \times \mathbb{Z}^I$-grading  by  defining the degrees of generators  as follows.
\begin{eqnarray*}
& &deg(E_i)=(i,0),\quad  deg(K_i)=(i,i)=deg(K'_i),\\
& &\ deg(F_i)=(0,i),\quad deg(K^{-1}_i)=(-i,-i)=deg(K'^{-1}_i).
\end{eqnarray*}
On $U_{v,t}$, we define a new multiplication $``\ast"$ by
\begin{equation}\label{eq63}
  x \ast y=t^{-[ |x|,\ |y|]'}xy,
\end{equation}
for any homogenous elements $x, y \in U_{v,t}$.
Since $[,]'$ is a bilinear form,
$(U_{v,t}, *)$ is an associative algebra over $\mathbb{Q}(v,t)$.
We define a multiplication, denoted by $``*"$, on $U_{v,t}\otimes U_{v,t}$ by
\begin{equation}\label{eq72}
(x\otimes y)*(x'\otimes y')=x*x' \otimes y*y'.
\end{equation}
This gives a new algebra structure on $U_{v,t}\otimes U_{v,t}$.
  $(U_{v,t}, *)$ has a Hopf algebra structure with the comultiplication $\Delta^*$, the counit $\varepsilon^*$ and the antipode $S^*$. The image of generators $E_i, F_i,K_i$ and $K_i^{-1}$ under the map $\Delta^*$ (resp. $\varepsilon^*$ and $S^*$) are the same as the ones under the map $\Delta$ (resp. $\varepsilon$ and $S$) defined in Section \ref{sec6}.

Under the new multiplication $``\ast"$,  the defining relations of $U_{v,t}$ in Section \ref{sec6} can be rewritten as follows.
\begin{eqnarray*}
  (R^*1)& & K_i^{\pm 1}\ast K^{\pm 1}_j=K^{\pm 1}_j\ast K_i^{\pm 1},\ \ K'^{\pm 1}_i\ast K'^{\pm 1}_j=K'^{\pm 1}_j\ast K'^{\pm 1}_i,\\
     & & K_i^{\pm 1}\ast K'^{\pm 1}_j=K'^{\pm 1}_j\ast K_i^{\pm 1},\ \ K_i^{\pm 1}\ast K_i^{\mp 1}=1=K'^{\pm 1}_i\ast K'^{\mp 1}_i.\\
  (R^*2)& &K_i\ast E_j\ast K^{-1}_i=v^{i \cdot j}E_j,\ \ \ \ \ \  K'_i\ast E_j\ast K'^{-1}_i=v^{-i \cdot j}E_j,\\
     & &K'_i\ast F_j\ast K'^{-1}_i=v^{i \cdot j}F_j,\ \ \ \ \ K_i\ast F_j\ast K^{-1}_i=v^{-i \cdot j}F_j.\\
  (R^*3)& &E_i\ast F_j-F_j\ast E_i=\delta_{ij}\frac{\widetilde{K}_i-\widetilde{K}'_i}{v_i-v^{-1}_i},\quad \forall i, j\in I.\\
  (R^*4)& & \sum_{p+p'=1-a_{ij}}(-1)^p\begin{bmatrix} 1-a_{ij}\\ p \end{bmatrix}_{v_i}E_i^{*p}\ast E_j\ast E_i^{*p'}=0,\quad {\rm if}\ i\not =j,\\
  & &\sum_{p+p'=1-a_{ij}}(-1)^p\begin{bmatrix} 1-a_{ij}\\ p \end{bmatrix}_{v_i}F_i^{*p}\ast F_j\ast F_i^{*p'}=0\quad {\rm if}\ i\not =j,
\end{eqnarray*}
 where $a_{ij}=2\frac{i\cdot j}{i\cdot i}$ and $E_i^{*p}=E_i*E_i*\cdots *E_i$ for $p$ copies. We notice that these relations are the specialization of (R1)-(R4) at $t=1$.

  {\it The one-parameter quantum algebra} $U_{v}(I, \cdot)$ associated to  $(I, \cdot)$ is the associative $\mathbb{Q}(v)$-algebra with 1 generated by symbols $E_i, F_i, K_i^{\pm 1}$, $ K_i'^{\pm 1}, \forall i\in I$ and  subject to relations (R*1)-(R*4).
  $U_{v}(I, \cdot)$ has a Hopf algebra structure with the comultiplication $\Delta_1$, the counit $\varepsilon_1$ and the antipode $S_1$.  The image of generators $E_i, F_i,K_i$ and $K_i^{-1}$ under the map $\Delta_1$ (resp. $\varepsilon_1$ and $S_1$) are the same as the ones under the map $\Delta$ (resp. $\varepsilon$ and $S$) defined in Section \ref{sec6}.

Let $U_{v,t}(I, \cdot):=U_{v}(I, \cdot)\otimes_{\mathbb{Q}(v)}\mathbb{Q}(v,t)$. The Hopf algebra structure on $U_{v}(I, \cdot)$ can be naturally extended to $U_{v,t}(I, \cdot)$.
From the above analysis, we have the following theorem.
\begin{thm}\label{thm14}
  If $(I, \cdot)$ is the Cartan datum associated to $\Omega$, then there is a Hopf-algebra isomorphism
 $$(U_{v,t}, \ast, \Delta^*, \varepsilon^*, S^*)\simeq (U_{v,t}(I,\cdot), \cdot, \Delta_1, \varepsilon_1, S_1),$$
 sending the generators in $U_{v,t}$ to the respective generators in $U_{v,t}(I,\cdot)$.
\end{thm}

\section{The canonical basis}

\subsection{The canonical basis of $\mathfrak{f}$}\label{sec2.5}

 Let ${}_{\mathfrak{A}}\mathfrak{f}$ be the $\mathbb{N}^I$-graded $\mathfrak{A}$-subalgebra of $\mathfrak{f}$ generated by $\theta_i^{(n)}$ for various $i\in I$ and $n\in \mathbb{N}$.
Let $\mathcal{B}$ be the subset of all elements $x$ in $\mathfrak{f}$ satisfying that
\begin{equation}\label{eq91}
  x\in {}_{\mathfrak{A}}\mathfrak{f},\quad \overline{x}=x, \quad (x,x) \in 1+v^{-1}\mathbb{Z}[[v^{-1}]],
\end{equation}
where $^{``- "}$ is defined in Section \ref{sec2.3} and (,) is defined in Proposition \ref{prop14}.
\begin{prop}\label{prop27}
  $\mathcal{B}\subset \mathbf{f}\otimes 1$.
\end{prop}
\begin{proof}
For any $x\in \mathcal{B}$, $x$ can be written as $x=\sum_{b\in \mathbf{f}\otimes 1} bt^{n_b}$. Moreover, there are only finite nonzero summands. So $max\{n_b\}$ exists, denoted by $n'$. Therefore, $(x,x)=t^{2n'}$ plus lower power terms. Since $(x,x) \in 1+v^{-1}\mathbb{Z}[[v^{-1}]]$, we have $n' \leq 0$.
Similarly, let $n''=min\{n_b\}$. Then $(x,x)=t^{2n''}$ plus higher power terms. Since $(x,x) \in 1+v^{-1}\mathbb{Z}[[v^{-1}]]$, we have $n'' \geq 0$. Therefore $n_b=0$ for all $b\in \mathbf{f}\otimes 1$.
Proposition follows.
 \end{proof}

Recall that a signed basis of an algebra $M$ is a subset, say $B$, of $M$ such that $B=B'\cup (-B')$ for some basis $B'$ of $M$.

\begin{thm}\label{thm9}
\begin{itemize}
  \item[(a)] $\mathcal{B}$ is a signed basis of the $\mathfrak{A}$-algebra ${}_{\mathfrak{A}}\mathfrak{f}$ and that of the $\mathbb{Q}(v,t)$-algebra $\mathfrak{f}$;
  \item[(b)] $(b,b')\in \delta_{bb'}+ v^{-1}\mathbb{Z}[[v^{-1}]]$, for any $b', b \in \mathcal{B}$.
\end{itemize}
\end{thm}

\begin{proof} By Proposition \ref{prop23}, Proposition \ref{prop27} and Theorem 14.2.3 in \cite{Lusztigbook}, Part (b) holds. Moreover, $\mathcal{B}$ is a signed basis of $\mathcal{A}$-module ${}_{\mathcal{A}}\mathbf{f}\otimes 1$, where $\mathcal{A}=\mathbb{Z}[v^{\pm 1}]$ and ${}_{\mathcal{A}}\mathbf{f}$ is the $\mathcal{A}$-subalgebra of $\mathbf{f}$ generated by $\theta_i^n/[n]_{v_i}^!$. Since ${}_{\mathfrak{A}}\mathfrak{f}=({}_{\mathcal{A}}\mathbf{f}\otimes 1)\otimes_{\mathcal{A}}\mathfrak{A}$, Part (a) follows.
 \end{proof}

We call $\mathcal{B}$ the {\it canonical signed basis} of $\mathfrak{f}$.

For any $i \in I$ and $n\in \mathbb{Z}_{\geq 0}$, let $\mathcal{B}_{i, \geq n}=\mathcal{B} \bigcap\theta_i^n\mathfrak{f}$ and $\mathcal{B}_{i, n}=\mathcal{B}_{i, \geq n}\setminus \mathcal{B}_{i, \geq n+1}$. Then we have a parition $\mathcal{B}_{i, \geq n}=\coprod_{n'\geq n}\mathcal{B}_{i, n'}$.

\begin{prop}\label{prop26}
If $b\in \mathcal{B}_{i,0}$, then there is a unique element $b' \in \mathcal{B}_{i,n}$ such that $t^{-n[i,|b|]}\theta_i^{(n)}b=b'$ plus an $\mathfrak{A}$-linear combination of elements in $\mathcal{B}_{i,\geq n+1}$. Moreover, there is a bijection $\pi_{i,n}: \mathcal{B}_{i,0}\rightarrow \mathcal{B}_{i,n}$ sending $b$ to $b'$.
\end{prop}

 \begin{proof} By Proposition \ref{prop27}, Proposition \ref{prop28} and Theorem 14.3.2(e) in \cite{Lusztigbook}, there is a unique 1-1 correspondence between $\mathcal{B}_{i,0}$ and $\mathcal{B}_{i,n}$ such that $\frac{\theta_i^{*n}}{[n]_{v_i}^!}*b=b'$ plus an $\mathcal{A}$-linear combination of elements in $\mathcal{B}_{i,\geq n+1}$, where $\theta_i^{*n}=\theta_i *\theta_i*\cdots *\theta_i$ for $n$ copies. By (\ref{eq89}), $\theta_i^{(n)}=\frac{\theta_i^{*n}}{[n]_{v_i}^!}$ and $\frac{\theta_i^{*n}}{[n]_{v_i}^!}*b=t^{-n[i,|b|]}\theta_i^{(n)}b$. Proposition follows.
 \end{proof}

Given any $\nu\in \mathbb{N}^I$, we define a subset $\mathfrak{B}_{\nu}$ of $\mathcal{B}$ by induction on ${\rm tr}(\nu)$. Let $\mathfrak{B}_0=\{1\}$. If ${\rm tr}(\nu) >0$, we set
$$\mathfrak{B}_{\nu}=\cup_{i \in I, n>0, \nu_i\geq n}\pi_{i,n}(\mathfrak{B}_{\nu-ni}\cap \mathcal{B}_{i,0}),$$
where $\pi_{i,n}$ is in Proposition \ref{prop26}.
Let
\begin{equation}\label{eq100}
  \mathfrak{B}=\sqcup_{\nu \in \mathbb{N}^I} \mathfrak{B}_{\nu}.
\end{equation}
The following theorem is an analogue of Theorem 14.4.3 in \cite{Lusztigbook}.
\begin{thm}\label{thm11}
  \begin{itemize}
    \item[(a)] $\mathcal{B}=\mathfrak{B} \cup (-\mathfrak{B})$;
    \item[(b)] For any $\nu \in \mathbb{N}^I$, $\mathfrak{B}_{\nu} \cap (-\mathfrak{B}_{\nu})=\varnothing$;
    \item[(c)] For any $\nu \in \mathbb{N}^I$, $\mathfrak{B}_{\nu}$ is a basis of the $\mathfrak{A}$-algebra ${}_{\mathfrak{A}}\mathfrak{f}_{\nu}$ and a basis of the $\mathbb{Q}(v,t)$-algebta $\mathfrak{f}_{\nu}$;
    \item[(d)] $\mathfrak{B}$ is a basis of the $\mathfrak{A}$-algebra ${}_{\mathfrak{A}}\mathfrak{f}$ and a basis of the $\mathbb{Q}(v,t)$-algebra $\mathfrak{f}$.
         % which is called the canonical basis of $\mathfrak{f}$.
  \end{itemize}
\end{thm}
\begin{proof}
   By definition of $\pi_{i,n}$ and $\mathcal{B}_{i.\geq n}$, we have $\mathfrak{B}\cup (-\mathfrak{B})\subset \mathcal{B}$.
   For any $\nu\in \mathbb{N}^I$ and any $x\in \mathcal{B}_{\nu}$, we are going to show that either $x\in \mathfrak{B}_{\nu}$ or $-x\in \mathfrak{B}_{\nu}$ by induction on ${\rm tr}(\nu)$.
   The case that ${\rm tr}(\nu)=0$ is trivial since $\mathfrak{B}_0=\{ 1\} $. Now assume that this statement is true for any $y \in \mathcal{B}$ with ${\rm tr}(|y|)<{\rm tr}(\nu)$.

  Since we have a partition $\mathcal{B}=\sqcup_{n'\geq 0}\mathcal{B}_{i,n'}$, $x\in \mathcal{B}_{i,m}$ for some $m \geq 0$.
  By Proposition \ref{prop26}, there exists $x' \in \mathcal{B}_{i,0}$ such that $x=\pi_{i,m}(x')$. Moreover, $x' \in \mathcal{B}_{\nu-mi}$.
  By induction assumption, either $x'\in \mathfrak{B}_{\nu-mi}$ or $-x'\in \mathfrak{B}_{\nu-mi}$.
   Therefore $x'\in \mathfrak{B}_{\nu-mi}\cap \mathcal{B}_{i,0}$ or $-x'\in \mathfrak{B}_{\nu-mi}\cap \mathcal{B}_{i,0}$. This implies that $x\in \mathfrak{B}_{\nu}$ or $-x\in \mathfrak{B}_{\nu}$. Part (a) follows.

 Part (b) is trivial for ${\rm tr}(\nu)=0$. For any $x\in \mathfrak{B}_{\nu}$, by the definition of $\mathfrak{B}_{\nu}$, there exists $x'\in \mathfrak{B}_{\nu-ni}$ for some $n\in \mathbb{N}$ and $i\in I$ such that $x=\pi_{i,n}(x')$. If $-x\in \mathfrak{B}_{\nu}$, then $-x'\in \mathfrak{B}_{\nu-ni}$. This is a contradiction by an induction on ${\rm tr}(\nu)$.

 Since $\mathcal{B}$ is a signed basis, part (c) follows from part (a) and (b).
 Part (d) follows from part (c).
\end{proof}

\begin{defn}
 The set $\mathfrak{B}$ defined in (\ref{eq100}) is called the canonical basis of $\mathfrak{f}$.
\end{defn}
 Recall that $\phi: \mathfrak{f}\rightarrow (\mathbf{f}_{v,t},\odot)$ is the algebra isomorphism in Theorem \ref{thm16}. Let $\mathfrak{B}(\phi)$ be the basis of $\mathfrak{f}$ such that the image of $\mathfrak{B}(\phi)$ under the map $\phi$ is the canonical basis of $\mathbf{f}_{v,t}$ defined in Theorem 14.4.3 in \cite{Lusztigbook}. Both $\mathfrak{B}$ and $\mathfrak{B}(\phi)$ consist of elements in $\mathfrak{f}$ satisfying (\ref{eq91})  by  Propositions \ref{prop23} and \ref{prop27}. Since $\mathfrak{B}_0=\{1\}=\mathfrak{B}_0(\phi)$, where  $\mathfrak{B}_0$ is the subset of $\mathfrak{B}$ consisting of all degree 0 elements, by the uniqueness of $\mathfrak{B}$, we have the following corollary.
\begin{cor}\label{cor8}
  $\mathfrak{B}(\phi)=\mathfrak{B}$. In other words, the canonical basis of $\mathfrak{f}$ is the same as that of $\mathbf{f}$ up to a 2-cocycle deformation. Moreover, if the associated Cartan data of $\Omega$ and $\Omega'$ are the same, then the canonical bases of $\mathfrak{f}$ and $\mathfrak{f}(\Omega')$ are the same if we present both elements by the multiplication $``*"$ in (\ref{eq89}).
\end{cor}

%In the following, we give some examples of the canonical basis of $\mathfrak{f}$.

\begin{ex}
 Let $I=\{i\}$ and $\Omega_{ii}=1$, then
  $\mathfrak{B}=\{\theta_i^{(n)}\ |\ n\in \mathbb{N}\}.$
\end{ex}

\begin{ex}
Let $I=\{i,j\}$ and $\Omega_{ii}=\Omega_{jj}=1, \Omega_{ij}=-1$, $\Omega_{ji}=0$. For any $a,b,c \in\mathbb{N}$ such that $a+c \leq b$, we set
  $$\mathfrak B_1=\{ t^{-a(b+c)}\theta_i^{(a)}\theta_j^{(b)}\theta_i^{(c)}\},\quad \mathfrak B_2=\{ t^{-a(b+c)}\theta_j^{(c)}\theta_i^{(b)}\theta_j^{(a)}\}.$$
By Section 14.5.4 in \cite{Lusztigbook},  $\theta_i^{(a)}\theta_j^{(b)}\theta_i^{(c)}=\theta_j^{(c)}\theta_i^{(b)}\theta_j^{(a)}$ if $b=a+c$. By identifying these two elements,
  $\mathfrak{B}=\mathfrak  B_1\cup \mathfrak  B_2.$
 \end{ex}

\subsection{The canonical basis of $L(\lambda, \epsilon)$}

Let $U_{v,t}^-$ be the negative part of $U_{v,t}$ generated by $F_i$ for all $i \in I$.
As  shown in Corollary ~\ref{cor7}, the algebra $U^-_{v, t}$ can be identified with  the algebra  $\mathfrak f$  by sending the generator $F_i$ to $\theta_i$ for any $i\in I$.
By abuse of notation, we denote by $\mathfrak{B}$ the image of the canonical  basis in $\mathfrak  f$ under the identification. For any pair
 $(\lambda, \epsilon) \in \mathbb{N}^I\times \mathbb{Q}(v,t)^I$ with $\epsilon \neq 0$, there exists a $U_{v,t}$-module $L(\lambda, \epsilon)$ containing a nonzero vector $\xi_0\in L(\lambda, \epsilon)$ and subject to
\begin{itemize}
 \item[(a)]$E_i \xi_0=0, K_i\xi_0=\epsilon_iv^{\lambda_i}\xi_0$ and $K_i'\xi_0=\epsilon_iv^{-\lambda_i}\xi_0$ for all $i \in I$,

 \item[(b)] The map $\varrho: U_{v,t}^- \rightarrow L(\lambda, \epsilon)$ given by $z\mapsto z\xi_0$ is surjective and its kernel is $\sum_{i\in I}U_{v,t}^-F_i^{\lambda_i+1}$.
\end{itemize}
Let $\mathfrak{B}(\lambda, \epsilon)=\varrho(\mathfrak{B} \setminus ((\sum_{i \in I}U_{v,t}^-F_i^{\lambda_i+1})\cap \mathfrak{B}))$.

\begin{prop}\label{prop29}
  (a) For any $\lambda \in \mathbb{N}^I$, the intersection $(\sum_{i \in I}\theta_i^{\lambda_i}\mathfrak{f})\cap \mathfrak{B}$ is a $\mathbb{Q}(v,t)$-basis of $\sum_{i \in I}\theta_i^{\lambda_i}\mathfrak{f}$.

  (b) For any $\lambda\in \mathbb{N}^I$, the intersection $(\sum_{i \in I}\mathfrak{f}\theta_i^{\lambda_i})\cap \mathfrak{B}$ is a $\mathbb{Q}(v,t)$-basis of $\sum_{i \in I}\mathfrak{f}\theta_i^{\lambda_i}$.
\end{prop}

\begin{proof} By Corollary 11.8 in \cite{Lusztigquiver} and Theorem \ref{thm16}.
\end{proof}

By Proposition \ref{prop29} and the identification of $U_{v,t}^-$ with $\mathfrak{f}$, we have that
%\begin{cor}\label{cor9}
  $\mathfrak{B}(\lambda, \epsilon)\subset L(\lambda, \epsilon)$  is a $\mathbb{Q}(v,t)$-basis of $L(\lambda, \epsilon)$.
%\end{cor}
\begin{defn}
   $\mathfrak{B}(\lambda, \epsilon)$  is called the canonical basis of $L(\lambda, \epsilon)$.
\end{defn}

\subsection{Positivity}\label{sec5.6}
%{Relationship between $\mathfrak{K}$ and ${}_{\mathfrak{A}}\mathfrak{f}$}
Recall that to the matrix $\Omega$ of symmetric type, we have constructed an algebra $\mathfrak{f}$ in Section \ref{sec2.3} and an algebra $\mathfrak{K}$ in Section \ref{sec3.8}.

\begin{thm}\label{thm8}
  The assignment $\theta_i^{(n)}\mapsto \mathfrak{L}_{ni}$ gives a twisted bialgebra isomorphism $\chi: {}_{\mathfrak{A}}\mathfrak{f}
  \simeq \mathfrak{K}$. Moreover, $\chi^{-1}(\widetilde{\mathfrak{B}})$ is the canonical basis of $\mathfrak{f}$ in Theorem \ref{thm11}, where $\widetilde{\mathfrak{B}}$ is the set of all isomorphism classes of simple perverse sheaves of weight 0.
\end{thm}
\begin{proof}
  The proof of the first part is the same as the proof of Theorem 13.2.11 in \cite{Lusztigbook}. We now show the second part. By Property 8.1.10 (d) in \cite{Lusztigbook}, (\ref{eq67}) and (\ref{eq91}), we have $\chi^{-1}(\widetilde{\mathfrak{B}})\subset \mathcal{B}$.
Let $\widetilde{\mathcal{B}}_{i,n}=\chi(\mathcal{B}_{i,n})$, where $\mathcal{B}_{i,n}$ is defined in Section \ref{sec2.5}.
For any $b\in \mathfrak{f}$, we write $\widetilde{b}=\chi(b)$. Let $\widetilde{\pi}_{i,n}=\chi \pi_{i,n}\chi^{-1}$, where $\pi_{i,n}$ is defined in Proposition \ref{prop26}.

We claim that $\widetilde{\pi}_{i,n}: \widetilde{\mathcal{B}}_{i,0}\rightarrow \widetilde{\mathcal{B}}_{i,n}$ preserves weights.
In fact, for any $\widetilde{b}\in \widetilde{\mathcal{B}}_{i,0}$, by Theorem \ref{thm8} and Proposition \ref{prop2}, $\widetilde{\pi}_{i,n}(\widetilde{b})$ is a direct summand of $\mathfrak{Ind}^{ni+|b|}_{ni, |b|}(\mathfrak{L}_{ni}\boxtimes \widetilde{b}(-\frac{n[i,|b|]}{2}))$. By Proposition \ref{prop3}, ${\rm wt}(\widetilde{\pi}_{i,n}(\widetilde{b}))={\rm wt}(\widetilde{b})$.

By the construction of $\mathfrak{B}$, all complexes whose isomorphism classes are in $\chi^{-1}(\mathfrak{B})$ have weight 0. Theorem follows.
\end{proof}

From  Theorem 14.4.13 in \cite{Lusztigbook}, Theorem \ref{thm16} and Proposition \ref{prop23}, we have
\begin{thm}{\rm (Positivity)}\label{thm13}
If $\Omega_{ii}=1$ for all $i\in I$, then we have
  \begin{itemize}
    \item[(a)]
    $bb'=\sum\limits_{b''\in \mathfrak{B}, n\in \mathbb{Z}}c_{b,b',b'',n}v^nt^{[|b|,|b'|]}b''$
   such that $c_{b,b',b'',n}\in \mathbb{N}$ are zero except for finitely many $b''$ and $n$ for all $ b,b'\in \mathfrak{B}$; \vspace{6pt}

    \item[(b)] $r(b)=\sum\limits_{b',b''\in \mathfrak{B}, n\in \mathbb{Z}} d_{b,b',b'',n}v^nt^{-[|b'|,|b''|]}b'\otimes b''$
    such that $ d_{b,b',b'',n}\in \mathbb{N}$ are zero except for finitely many $b', b''$ and $n$ for all $b\in \mathfrak{B}$;\vspace{6pt}

    \item[(c)]  $(b,b')=\sum\limits_{n\in \mathbb{N}} g_{b,b',n}v^{-n}$
    such that $g_{b,b',n}\in \mathbb{N}$ for all $b,b'\in \mathfrak{B}$.
  \end{itemize}
\end{thm}

 The structure constants with respect to the canonical bases between $\mathfrak{f}$ and Lusztig's algebra $\mathbf{f}$ differ by a certain power of $t$ due to Theorem \ref{thm16}. In particular, the specialization of the structure constants of $\mathfrak{f}$ with respect to $\mathfrak{B}$ at $t=1$ gives the structure constants of $\mathbf{f}$ with respect to the canonical basis of $\mathbf{f}$.

\section{A categorification of ${}_{\mathfrak{A}}\mathfrak{f}$}
We shall give a categorification of ${}_{\mathfrak{A}}\mathfrak{f}$ for arbitrary $\Omega$ based on a categorification of the integral form  ${}_{\mathcal{A}}\mathbf{f}$ of Lusztig's algebra $\mathbf f$.
The followings are some examples of categorifications of ${}_{\mathcal{A}}\mathbf{f}$.

 \begin{ex}
The triple  $(\oplus_{\nu \in \mathbb{N}^I} \mathcal{Q}_{\nu},\Ind,\Res)$ constructed in
  ~\cite[Chapter 9]{Lusztigbook} is a categorification of ${}_{\mathcal{A}}\mathbf{f}$. Note that $\mathcal{Q}_{\nu}=\mathfrak{Q}_{\nu}^{\leq 0}\cap \mathfrak{Q}_{\nu}^{\geq 0}$.
 \end{ex}
 \begin{ex}
   The triple $(\oplus_{\nu\in \mathbb{N}^I}R_{\nu}$-$mod,\Ind,\Res)$ in \cite{Khovanov2009diag} is a categorification of ${}_{\mathcal{A}}\mathbf{f}$, where $R_{\nu}$-$mod$ is a category of certain projective modules.
 \end{ex}

We fix a categorification $(\oplus_{\nu\in \mathbb{N}^I}\mathcal{Q}_{\nu}, \Ind, \Res)$ of $\mathcal{A}$-bialgebra ${}_{\mathcal{A}}\mathbf{f}$.
Given any $n\in \mathbb{Z}$, for each $\nu\in \mathbb{N}^I$, let $\mathcal{Q}_{n,\nu}$ be a category which is identical to $\mathcal{Q}_{\nu}$. We identify $\mathcal{Q}_{\nu}$ with $\mathcal{Q}_{0,\nu}$. For a fix $\nu\in \mathbb{N}^I$, the category $\mathcal{Q}_{n,\nu}$ are all identical to each other for different $n\in \mathbb{Z}$.
 Denote by $\mathcal{T}: \mathcal{Q}_{n-1,\nu} \rightarrow \mathcal{Q}_{n,\nu}$ the identity functor. We also denote by $\mathcal{T}^n: \mathcal{Q}_{k,\nu} \rightarrow \mathcal{Q}_{n+k,\nu}$ the composition functor of $\mathcal{T}$.

Let $\iota_{\nu}: \mathcal{Q}_{\nu} \rightarrow \oplus_{\nu}\mathcal{Q}_{\nu}$ and $p_{\nu}: \oplus_{\nu}\mathcal{Q}_{\nu} \rightarrow\mathcal{Q}_{\nu}$ be the natural embedding and projection functor, respectively. For any $\nu=\tau+\omega$, denote $\Ind^{\nu,0}_{\tau, \omega}=p_{\nu}\circ \Ind \circ (\iota_{\tau}\times \iota_{\omega})$ and $\Res^{\nu,0}_{\tau, \omega}= (p_{\tau}\times p_{\omega})\circ\Res\circ \iota_{\nu}$. We define
\begin{equation}\label{eq43}
\Ind^{\nu,n,m}_{\tau, \omega}:  \mathcal{Q}_{n, \tau}\times  \mathcal{Q}_{m, \omega}\rightarrow  \mathcal{Q}_{n+m, \nu},\quad
(L,M) \mapsto \mathcal{T}^{n+m}\circ \Ind^{\nu,0}_{\tau, \omega}(\mathcal{T}^{-n}L, \mathcal{T}^{-m}M),\ {\rm and}
\end{equation}
\begin{equation}\label{eq44}
\Res^{\nu,n,m}_{\tau, \omega}:  \mathcal{Q}_{n+m, \nu} \rightarrow  \mathcal{Q}_{n, \tau}\times  \mathcal{Q}_{m, \omega},\quad
L \mapsto (\mathcal{T}^{n}\times \mathcal{T}^m) \circ \Res^{\nu,0}_{\tau, \omega}\circ \mathcal{T}^{-(n+m)}L.
\end{equation}

  Let $\mathfrak{Q}_{\nu}=\oplus_{n\in \mathbb{Z}} \mathcal{Q}_{n,\nu}$ and  $\mathfrak{Q}=\oplus_{\nu\in \mathbb{N}^I}\mathfrak{Q}_{\nu}$. Define a $\mathbb{Z}[t^{\pm 1}]$-action on the split Grothendieck group $K_0(\mathfrak{Q}_{\nu})$ of $\mathfrak{Q}_{\nu}$ by
  $$t\cdot[L]=[\mathcal{T}(L)],$$
  where $[L]$ is the isomorphism class of $L$.
  Since $K_0(\mathcal{Q}_{n,\nu})$ carries an $\mathcal{A}$-module structure for each pair $(n,\nu)$, the above action defines an $\mathfrak{A}$-module structure on $K_0(\mathfrak{Q}_{\nu})$.

 Given a functor $\mathfrak{F}$ between any two categories, we denote by $[\mathfrak{F}]$ the induced map between the corresponding Grothendieck groups.
  By (\ref{eq43}) and (\ref{eq44}),  we have
  \begin{equation}\label{eq45}
    [\Ind^{\nu,n,m}_{\tau, \omega}] \circ (t^n \times t^m)=t^{n+m} \circ [\Ind^{\nu,0}_{\tau, \omega}],\  {\rm and}\ (t^n \times t^m)\circ [\Res^{\nu,n,m}_{\tau, \omega}]=[\Res^{\nu,0}_{\tau, \omega}] \circ t^{n+m}.
  \end{equation}

  For any $\nu=\tau+\omega$, we define functors
  \begin{eqnarray}\label{eq60}
&\mathfrak{Ind}^{\nu,n,m}_{\tau, \omega}: \mathfrak{Q}_{\tau}\times \mathfrak{Q}_{\omega} \rightarrow  \mathfrak{Q}_{\nu},\quad
(L,M) \mapsto  \mathcal{T}^{[ \tau,\omega]}\circ \Ind^{\nu,n,m}_{\tau, \omega}(L,M),\ {\rm and} \nonumber
  \end{eqnarray}
\begin{eqnarray}\label{eq61}
&\mathfrak{Res}^{\nu,n,m}_{\tau, \omega}:\mathfrak{Q}_{\nu}  \rightarrow  \mathfrak{Q}_{\tau}\times \mathfrak{Q}_{\omega}, \quad
L \mapsto  \mathcal{T}^{-[ \tau,\omega]}\circ \Res^{\nu,n,m}_{\tau, \omega}L, \nonumber
  \end{eqnarray}
where [,] is defined in (\ref{eq48}). By assembling $\mathfrak{Ind}^{\nu,n,m}_{\tau, \omega}$ together, we have a functor $\mathfrak{Ind}: \mathfrak{Q}\otimes \mathfrak{Q}\rightarrow \mathfrak{Q}$. Similarly, we have a functor $\mathfrak{Res}=\oplus_{\tau+\omega=\nu}\mathfrak{Res}^{\nu,n,m}_{\tau, \omega}$.

  \begin{thm}\label{thm15}
  If $(\mathcal{Q}, \Ind, \Res)$ is a categorification of the $\mathcal{A}$-bialgebra ${}_{\mathcal{A}}\mathbf{f}$, then
$(\mathfrak{Q}, \mathfrak{Ind}, \mathfrak{Res})$ is a categorification of the $\mathfrak{A}$-bialgebra ${}_{\mathfrak{A}}\mathfrak{f}$.
  \end{thm}
  \begin{proof} Recall that the pair $(*, \widetilde{r})$ defined in (\ref{eq89}) and (\ref{eq68}) gives a new bialgebra structure on ${}_{\mathfrak{A}}\mathfrak{f}$.

    Since $(\mathcal{Q}, \Ind, \Res)$ is a categorification of ${}_{\mathcal{A}}\mathbf{f}$, there exists a bialgebra isomorphism $\chi: {}_{\mathcal{A}}\mathbf{f} \rightarrow K_0(\mathcal{Q})$.
    Therefore, $\chi\otimes 1: {}_{\mathcal{A}}\mathbf{f}\otimes_{\mathcal{A}}\mathfrak{A} \rightarrow K_0(\mathcal{Q})\otimes_{\mathcal{A}}\mathfrak{A}$ is a bialgebra isomorphism.
    The bialgebra structure on ${}_{\mathcal{A}}\mathbf{f}\otimes_{\mathcal{A}}\mathfrak{A}$ (resp. $K_0(\mathcal{Q})\otimes_{\mathcal{A}}\mathfrak{A}$) can be obtained by field extension.

  Recall that there is a bialgebra isomorphism $\rho: {}_{\mathcal{A}}\mathbf{f}\otimes_{\mathcal{A}}\mathfrak{A}\rightarrow ({}_{\mathfrak{A}}\mathfrak{f}, *, \widetilde{r})$ (see Proposition \ref{prop28}).
  Consider the $\mathfrak{A}$-linear map
   \begin{equation*}
   \psi: K_0(\mathcal{Q})\otimes_{\mathcal{A}}\mathfrak{A}\rightarrow (K_0(\mathfrak{Q}), [\Ind], [\Res]),\quad
   L\otimes t^n \mapsto t^n\cdot L.
   \end{equation*}
 We want to show that $\psi$ is a bialgebra isomorphism. It is a bijective map as an $\mathfrak{A}$-linear map. So it is enough to show that it is a bialgebra homomorphism. Firstly, $\psi$ is an algebra homomorphism, since
    \begin{eqnarray*}
&\psi((L\otimes t^n)(M\otimes t^m))=\psi([\Ind](L,M)\otimes t^{n+m})
=t^{n+m}\cdot [\Ind](L,M)\\
&=[\Ind](t^nL, t^mM)
=[\Ind](\psi(L\times t^n),\psi(M\times t^m)).\hspace{70pt}
    \end{eqnarray*}
   Secondly, $\psi$ is a coalgebra homomorphism, because
    \begin{eqnarray*}
     ( \psi\times \psi)([\Res](L\otimes t^n)) = ( \psi\times \psi)([\Res](L)\otimes t^n)
     =(t^n \otimes 1)\cdot [\Res](L).
    \end{eqnarray*}
    On the other hand, we have
    $$[\Res](\psi(L\otimes t^n)) =[\Res](t^n \cdot L )=(t^n\otimes 1) \cdot [\Res](L).$$
   Therefore, we have the following diagram,
   $$\xymatrix{{}_{\mathcal{A}}\mathbf{f}\otimes _{\mathcal{A}}\mathfrak{A} \ar[r]^-{\chi\otimes 1} \ar[d]^{\rho} & K_0(\mathcal{Q})\otimes_{\mathcal{A}}\mathfrak{A} \ar[d]^-{\psi}\\
   ({}_{\mathfrak{A}}\mathfrak{f}, *, \widetilde{r}) \ar[r]^-{\widetilde{\chi}} & (K_0(\mathfrak{Q}), [\Ind],[\Res]),
   }$$
   where $\widetilde{\chi}=\psi\circ (\chi\otimes 1) \circ \rho^{-1}$. Since $\psi,\ \chi\otimes 1,\  \rho^{-1}$ are all bialgebra isomorphisms.
   This forces $\widetilde{\chi}$ to be also a bialgebra isomorphism.

   Lastly, we show that  $\widetilde{\chi}: {}_{\mathfrak{A}}\mathfrak{f} \rightarrow (K_0(\mathfrak{Q}), \mathfrak{Ind},\mathfrak{Res})$ is also a bialgebra isomorphism. As a $\mathfrak{A}$-linear map, $\widetilde{\chi}$ is a bijective map. So it is enough to show that $\widetilde{\chi}$ is a bialgebra homomorphism. $\widetilde{\chi}$ is an algebra homomorphism, since,  for any homogeneous elements $L,M\in {}_{\mathfrak{A}}\mathfrak{f}$, we have
   \begin{eqnarray*}
\widetilde{\chi}(LM)=t^{[|L|,|M|]}\widetilde{\chi}(L *M)
=t^{[|L|,|M|]}\Ind(\widetilde{\chi}(L),\widetilde{\chi}(M))
=\mathfrak{Ind}
   (\widetilde{\chi}(L),\widetilde{\chi}(M)).
   \end{eqnarray*}
  For any $L \in  {}_{\mathfrak{A}}\mathfrak{f}$, let us write $\widetilde{r}(L)=\sum L_1\otimes L_2$. Then we have,
  \begin{eqnarray*}
\widetilde{\chi}(r(L))=\widetilde{\chi}(t^{-[|L_1|,|L_2|]}L_1\otimes L_2)
=\sum t^{-[|L_1|,|L_2|]}(\widetilde{\chi}(L_1)\otimes \widetilde{\chi}(L_2))
=\mathfrak{Res}
   (\widetilde{\chi}(L)).\end{eqnarray*}
   This finishes the proof.
  \end{proof}

%\bibliographystyle{amsplain-nodash}
%%\bibliographystyle{plain}
%\bibliography{references}

\end{document}